\documentclass[amssymb,amsfonts,refcheck,11pt,verbatim,righttag]{amsart}
\usepackage{amssymb}
\def\colon{{:}\;}

\def\R {\Bbb R}

\newcommand{\beq}{\begin{equation}}
\newcommand{\eeq}{\end{equation}}

\newcommand{\ben}{\begin{eqnarray}}
\newcommand{\een}{\end{eqnarray}}

\newcommand{\bet}{\begin{eqnarray*}}
\newcommand{\eet}{\end{eqnarray*}}

\newtheorem{thm}{Theorem}[section]
\newtheorem{lem}[thm]{Lemma}
\newtheorem{pro}[thm]{Proposition}
\newtheorem{cor}[thm]{Corollary}
\newtheorem{ex}[thm]{Example}
\newtheorem{de}[thm]{Definition}
\newtheorem{re}[thm]{Remark}

\def\R {\Bbb R}
\def\N {\Bbb N}
\def\M {{\mathcal M}}

\def\E {{\mathcal E}}
\def\a{{\bf a}}
\def\b{{\bf b}}
\def\bt{{\bf t}}

\def\I{{\mathcal I}}

\def\bq{{\bf q}}

\def\C{{\mathcal C}}
\def\htop{h_{\rm top}}
\theoremstyle{plain}

\begin{document}
\baselineskip 14pt

\title
{Lyapunov spectrum of asymptotically sub-additive  potentials}

\author{De-Jun FENG}
\address{
Department of Mathematics\\
The Chinese University of Hong Kong\\
Shatin,  Hong Kong\\
}
\email{djfeng@math.cuhk.edu.hk}

\author{Wen Huang}
\address{
Department of Mathematics\\
 University of Science and Technology
of
China\\
 Hefei 230026, Anhui,  P. R. China\\
}
\email{wenh@mail.ustc.edu.cn }

\date{}

\begin{abstract}
For general asymptotically sub-additive potentials (resp. asymptotically additive potentials) on general topological dynamical systems,
we establish  some variational relations between the topological entropy of the level sets of
Lyapunov exponents, measure-theoretic entropies and  topological pressures
in this general situation.  Most of our results are obtained without
 the assumption of the existence of unique equilibrium measures or the differentiability of pressure functions. Some examples are constructed to illustrate the irregularity and the complexity of multifractal behaviors in the sub-additive
case and in the case that the entropy map that is not upper-semi continuous.
\end{abstract}

\maketitle
\setcounter{section}{0}
\section{Introduction}
\label{S-1}
\setcounter{equation}{0}
The present paper is devoted to the study of  the multifractal behavior of  Lyapunov exponents of asymptotically sub-additive potentials.
This  is mainly motivated by the recent works on the Lyapunov exponents of matrix products
\cite{FeLa02, Fen03, Fen08} and the Lyapunov exponents of differential maps on nonconformal repellers \cite{BaGe06}.

Before formulating  our results, we first give some notation and backgrounds.  We call $(X,T)$ a {\it topological dynamical system}
(for short TDS) if $X$ is
a compact metric space  and $T$ is a continuous map from $X$ to $X$.
 A sequence   $\Phi=\{\log \phi_n\}_{n=1}^\infty$ of functions on $X$ is said to be a {\it sub-additive potential} if
each $\phi_n$ is a continuous nonnegative-valued function on $X$ such that
\begin{equation}
\label{e-1.0}
 0\leq \phi_{n+m}(x)\leq \phi_n(x)\phi_m(T^nx),\qquad \forall\;
x\in X, \; m,n\in \N.
\end{equation}
More generally,   $\Phi=\{\log \phi_n\}_{n=1}^\infty$ is said to be an {\it asymptotically sub-additive potential} if for any $\epsilon>0$, there exists a sub-additive potential $\Psi=\{\log \psi_n\}_{n=1}^\infty$ on $X$ such that
$$
\limsup_{n\to \infty}\frac{1}{n}\sup_{x\in X} |\log \phi_n(x)-\log \psi_n(x)|\leq \epsilon,
$$
where we take the convention $\log 0-\log 0=0$.
Furthermore $\Phi$ is called an {\it asymptotically additive potential}
if both $\Phi$ and $-\Phi$ are asymptotically sub-additive, where $-\Phi$ denotes $\{\log (1/\phi_n)\}_{n=1}^\infty$.
In particular,  $\Phi$ is called {\it additive} if each $\phi_n$ is a continuous positive-valued function so that
$\phi_{n+m}(x)=\phi_n(x)\phi_m(T^nx)$ for all $x\in X$ and $m,n\in \N$; in this case, there is a
continuous real function $g$
on $X$ such that $\phi_n(x)=\exp(\sum_{i=0}^{n-1}g(T^ix))$ for each $n$.

Let $\Phi=\{\log \phi_n\}_{n=1}^\infty$ be an  asymptotically
sub-additive potential on $X$. For any $x\in X$, we define
\begin{equation}
\label{e-1.1} \lambda_{\Phi}(x)=\lim_{n\to \infty}\frac{\log
\phi_n(x)}{n}
\end{equation}
and call it the {\it Lyapunov exponent of $\Phi$ at $x$}, provided
that  the limit exists. Otherwise we use
$\overline{\lambda}_\Phi(x)$ and ${\underline{\lambda}}_\Phi(x)$ to
denote the upper and lower limits respectively.   It can be derived
from  Kingman's sub-additive ergodic theorem (cf. \cite{Wal-book},
p. 231) that, for any $\mu\in \E(X,T)$,
$$
\lambda_\Phi(x)=\Phi_*(\mu) \qquad \mbox{ for } \mu \mbox{-a.e. }
x\in X,
$$
where  $\E(X,T)$ denotes the space of ergodic $T$-invariant
Borel probability  measures on $X$ and
\begin{equation} \label{e-1.2} \Phi_*(\mu):=\lim_{n\to \infty}\int
\frac{\log \phi_n(x)}{n}\; d\mu(x).
\end{equation}
 This limit always exists and takes values in $\R\cup \{-\infty\}$. (For details, see Proposition \ref{pro-1195}.)

In this paper we are mainly concerned with the distribution of the Lyapunov exponents of $\Phi$.
More precisely, for any $\alpha\in \R$,  define
\begin{equation}
\label{e-1.4} E_\Phi(\alpha)=\{x\in X:\; \lambda_\Phi(x)=\alpha\},
\end{equation}
which is called the {\it $\alpha$-level set} of $\lambda_\Phi$.
 We shall study the topological entropy $\htop(T, E_\Phi(\alpha))$ of
 $E_\Phi(\alpha)$ when $\alpha$ varies (here we are using the notion of topological entropy for arbitrary subsets of a compact space,
 introduced by Bowen in \cite{Bow73}; see Section \ref{S-2.1}).
  This is a general concept of multifractal analysis proposed by Barreira, Pesin and Schmeling
  \cite{BPS97}, and it plays an important role in the dimension theory of dynamical systems \cite{Pes-book, Bar-book}. For convenience we call $\htop(T, E_\Phi(\alpha))$, as a function of $\alpha$, the {\it Lyapunov spectrum} of $\Phi$.

A key ingredient in the above study is the  topological pressure of $\Phi$. To introduce this concept, let $X$ be endowed with the metric $d$.
For any $n\in \N$, define a new  metric $d_n$ on $X$ by
\begin{equation}
\label{e-1.4'}
 d_n(x, y) =\max\left\{d\left(T^k(x), T^k(y)\right):\;k =0, \ldots,n-1 \right\}.
 \end{equation}
 For any $\epsilon>0$, a set $E\subseteq X$ is
said to be a {\it $(n,\epsilon)$-separated subset of $X$}  if
$d_n(x,y)>\epsilon$ for any two different points $x,y\in E$. For $\Phi=\{\log \phi_n\}_{n=1}^\infty$, we define
$$
P_n(T, \Phi,\epsilon) =\sup\left\{\sum_{x\in E}\phi_n(x):\; E \mbox{ is
a } (n,\epsilon) \mbox{-separated subset of } X\right\}.
$$
It is clear that $P_n(T, \Phi,\epsilon)$ is a decreasing function of
$\epsilon$. Define
$$
P(T,\Phi,\epsilon)=\limsup_{n\to \infty}\frac{1}{n}\log
P_n(T,\Phi,\epsilon)
$$
and $P(T,\Phi)=\lim_{\epsilon\to 0}P(T,\Phi,\epsilon)$. We call  $P(T,\Phi)$
 the {\it topological pressure of $\Phi$ with respect to
$T$} or, simply, the {\it topological pressure of $\Phi$}. If $\Phi$ is additive, $P(T,\Phi)$ recovers
the classical (additive) topological pressure introduced by Ruelle and Walters (cf. \cite[Chapter 9]{Wal-book}).

Let us return back to the study of the Lyapunov spectrum.
When $\Phi=\{\sum_{i=0}^{n-1}f\circ T^i\}_{n=1}^\infty$ is an additive potential, the Lyapunov exponent $\lambda_\Phi$ is just equal to
the Birkhoff average of $f$. In this case, the topological entropy
(or the Hausdorff dimension) of the level sets of Birkhoff averages  has been extensively studied in the recent two
decades (see, e.g.,
\cite{BMP92, BPS97, Oli99, BaSc00, FaFe00, FFW01, Kes01, FLW02, BSS02, Ols03, TaVe03, Fen03,
 CKS05, FLP08, BaMe07, BaMe08, FeSh08} and references therein).
It is well known (see, e.g.  \cite{FFW01, Fen03, Oli99}) that when $(X,T)$  is a transitive subshift of finite type and
$\Phi$ is an additive potential, then
\begin{equation}
\label{e-0404}
E_\Phi(\alpha)\neq \emptyset\Longleftrightarrow \alpha\in \Omega:=\{\Phi_*(\mu): \; \mu\in \M(X,T)\},
\end{equation} where $\M(X,T)$ denotes the space of $T$-invariant
Borel probability measures on $X$ and \begin{equation}
\label{e-i1}
\begin{split}
\htop(T,E_\Phi(\alpha))&=\sup\{h_\mu(T):\; \mu\in \M(X,T) \mbox{ with } \Phi_*(\mu)=\alpha\}\\
\mbox{}&=\inf\{P_\Phi(q)-\alpha q:\; q\in \R\}, \qquad \forall\; \alpha\in \Omega.
\end{split}
\end{equation}
where $h_\mu(T)$ denotes the measure-theoretic entropy of $\mu$, $P_\Phi(q):=P(T,q\Phi)$ and
$q\Phi$ denotes the potential $\{q\log \phi_n\}_{n=1}^\infty$.   The first variational relation in (\ref{e-i1})
has been extended  to any TDS satisfying the specification property \cite{TaVe03}.

Motivated by the study of the multifractal formalism associated to certain iterated function systems with overlaps,
 the Lyapunov spectrum of certain special sub-additive potentials $\Phi=\{\log \phi_n\}_{n=1}^\infty$ on subshifts of finite type  have been studied
 in \cite{FeLa02, Fen03, Fen08},  in which $\phi_n(x)=\|\prod_{i=0}^{n-1}M(T^ix)\|$, where $M$ is a continuous function on $X$ taking values in the set of
$d\times d$ matrices, and $\|\cdot\|$ denotes the operator norm. It
is known that in this general situation, (\ref{e-0404}) and
(\ref{e-i1}) may both fail. The following is an example taken from
\cite{Fen08}.

\noindent{\bf Example 1.1} {\it  Let $(X,T)$ be the  one-sided full shift over the alphabet $\{1,2,3,4\}$. Let
 $M(x)$ be  a matrix function on $X$  defined as
$M(x)=M_{x_1}$ for $x=(x_j)_{j=1}^\infty$, where $M_i$ ($1\leq
i\leq 4$) are  diagonal $4\times 4$ matrices given by
$$M_1=M_2=\mbox{\rm diag}(1,2,0,0), \ \ M_3=\mbox{\rm diag}(1,0,3,0),\
M_4=\mbox{\rm diag}(1,0,0,4).$$ It is easily checked that
\[
P_\Phi(q)=\left\{
\begin{array}{ll}
q\log 4, & \mbox{ if }q\geq 1\\
\log 4, & \mbox{ otherwise}
\end{array}
\right.
\]
and \begin{equation*}
\begin{split}
\{\alpha\in \R:\;  E_\Phi(\alpha)\neq \emptyset\}&=\{0,\log 2,\log 3,\log4\}\\
&\subsetneqq [0,\log 4]= \{\alpha\in \R:\; \Phi_*(\mu)
=\alpha\mbox{ for some }\mu\in \M(X,T)\}.
\end{split}
\end{equation*}
Furthermore, $E_\Phi(\log 3)$ is a singleton and thus
$$
\htop(T, E_\Phi(\log 3))=0<\log 4-{\log 3}=\inf
_{q\in \R}\{-q\log 3+P_\Phi(q)\}.
$$
}

We remark that under some additional assumptions
(e.g., positiveness or certain irreducibility) for the matrix function $M$,  (\ref{e-0404}) and (\ref{e-i1})
still hold \cite{FeLa02, Fen03, Fen08}. A natural question arises whether there exist some positive results without
any additional assumptions. This is one of the original motivations of this paper.

Indeed in this paper, we study the Lyapunov spectrum of {\it general} asymptotically sub-additive potentials and
asymptotically additive potentials on {\it general} TDS.
Under this setting, the multifractal behavior may be quite irregular. For instance, we can construct a TDS $(X, T)$ and
an additive potential $\Phi$ on $X$ such that
$$\htop(T,E_\Phi(\alpha))< \inf\{P_\Phi(q)-\alpha q:\; q\in \R\}\qquad \forall\; \alpha\in \Omega:=\{\Phi_*(\mu): \; \mu\in \M(X,T)\}.$$
  (See Example \ref{re-1.2}).
Nevertheless, we still have some positive results regarding the Lyapunov spectrum and its variational
relations to measure-theoretic entropies and topological pressures. Some more properties are obtained when the corresponding TDS
satisfies further  assumptions (e.g., upper semi-continuity of the entropy map).

To formulate our results,  for an asymptotically sub-additive potential $\Phi=\{\log \phi_n\}$ on a general TDS $(X,T)$,
 we define
\begin{equation}
\label{e-1.5} \overline{\beta}(\Phi)=\lim_{n\to \infty}\frac{1}{n}\log
\sup_{x\in X}\phi_n(x).
\end{equation}
The limit exists and takes values in  $\R\cup
\{-\infty\}$ (see Lemma \ref{lem-mf-sa}). However if  $\overline{\beta}(\Phi)=-\infty$, it is easy to see
that  for all $x\in X$, $\lambda_\Phi(x)=-\infty$. To avoid
trivialities we shall always assume that $\overline{\beta}(\Phi)>-\infty$.  For any $q>0$, let $q\Phi$ denote the sequence
$\left\{q\log \phi_n\right\}_{i=1}^\infty$ (which  clearly is
asymptotically sub-additive) and write $$P_\Phi(q)=P\left(T,q\Phi\right).$$ The
function $P_\Phi$ is called the {\it pressure function} of $\Phi$. When $\Phi$ is  asymptotically additive on $X$,
$P_\Phi$ can be defined over $(-\infty,\infty)$.

Our main results are Theorems \ref{thm-1.1}-\ref{thm-1.4} formulated as follows:
\begin{thm}
\label{thm-1.1} Let $(X,T)$ be a TDS and  $\Phi=\{\log \phi_n\}_{n=1}^\infty$ an asymptotically  sub-additive potential  on $X$ which  satisfies  $\overline{\beta}(\Phi)
>-\infty$. Then $E_{\Phi}(\overline{\beta}(\Phi))\neq \emptyset$ and
\begin{align*}
\htop(T,E_{\Phi}(\overline{\beta}(\Phi)))&= \sup \{ h_\mu(T):~
\mu\in \mathcal{M}(X,T),
~\Phi_*(\mu)=\overline{\beta}(\Phi)\}\\&=\sup \{ h_\mu(T):~ \mu\in
\E(X,T), ~ \Phi_*(\mu)=\overline{\beta}(\Phi)\}.
\end{align*}
\end{thm}

We emphasize  that the above theorem only deals with the specific value $\alpha=\overline{\beta}(\Phi)$, which is the largest possible value for $\lambda_\Phi$ (cf.
Lemma \ref{lem-mf-sa}).
\begin{thm}
\label{thm-1.2} Let $(X,T)$ be a TDS such that the topological entropy $\htop(T)$
is finite. Suppose that  $\Phi=\{\log \phi_n\}_{n=1}^\infty$ is an asymptotically  sub-additive potential  on $X$ which  satisfies  $\overline{\beta}(\Phi)
>-\infty$. Then the pressure function $P_\Phi(q)$ is a continuous real convex
function on $(0,\infty)$ with
$P^{\prime}_\Phi(\infty):=\lim_{q\to\infty}P_\Phi(q)/q=\overline{\beta}(\Phi)$.
Moreover,
\begin{itemize}
\item[(i)]For any $t>0$, if $\alpha=P^\prime_\Phi(t+)$ or
$\alpha=P^\prime_\Phi(t-)$, then
$$
\lim_{\epsilon\to
0}\htop\left(T,\bigcup_{\beta\in(\alpha-\epsilon,\alpha+\epsilon)}
E_\Phi(\beta)\right)=\inf_{q>0}\{P_\Phi(q)-\alpha q\}=P_\Phi(t)-\alpha t,
$$
 where $P^\prime_\Phi(t-)$ and $P^\prime_\Phi(t+)$ denote the
left and right derivatives of $P_\Phi$ at $t$, respectively.
  Moreover the first equality is also valid when $\alpha=P^\prime_\Phi(\infty)$.
\item[(ii)] For any $t>0$ and any $\alpha\in
[P^\prime_\Phi(t-),P^\prime_\Phi(t+)]$,
 $$
\inf_{q>0}\{P_\Phi(q)-\alpha q\}=\lim_{\epsilon\to 0}\sup\{h_\mu(T):\;
\mu\in \M(X,T),\;|\Phi_*(\mu)-\alpha|<\epsilon\}.
  $$
Furthermore the  above equality is valid  for
$\alpha=P^\prime_\Phi(\infty)$.

 \item[(iii)] For any  $\alpha\in
\left(\lim_{t\to 0+}P^\prime_\Phi(t-),P^\prime_\Phi(\infty)\right)$,
 $$
\inf_{q>0}\{P_\Phi(q)-\alpha q\}=\sup\{h_\mu(T):\; \mu\in
\M(X,T),\;\Phi_*(\mu)=\alpha\}.
  $$
\end{itemize}
\end{thm}

By convexity, $P_\Phi$ may fail to be differentiable on a set which
is at most countable, however the left and right derivatives of
$P_\Phi$ exist everywhere. We remark that $\htop(T)$ is finite for a
lot of  TDS's such as expansive maps on compact metric space and
Lipschitz continuous transformations on finite dimensional compact
metric spaces (see e.g. \cite[Section 3.2]{KaHa-book}), and
asymptotically $h$-expansive TDS's \cite{Mis76}. We need to mention
that Theorem \ref{thm-1.2}(i) only deals with the ``fuzzy'' level sets and it is not valid for the standard level sets $E_\Phi(\alpha)$. Indeed, there are examples  such that
$$\htop\left(T,E_\Phi(\alpha)\right)<\inf_{q>0}\{P_\Phi(q)-\alpha
q\}$$ for any $\alpha=P^\prime_\Phi(t+)$ or
$\alpha=P^\prime_\Phi(t-)$ with $t>0$ (see e.g. Example
\ref{re-1.2}). Nevertheless the results of Theorem \ref{thm-1.2} can be
improved if we add an additional assumption that the entropy map
$\mu\to h_\mu(T)$ is upper semi-continuous on $\M(X,T)$. More
precisely, we have

\begin{thm}
\label{thm-1.3} Under the condition of Theorem \ref{thm-1.2}, we
assume furthermore that the entropy map $\mu\to h_\mu(T)$ is upper
semi-continuous on $\M(X,T)$. Then
\begin{itemize}
\item[(i)]For any $t>0$, if $\alpha=P^\prime_\Phi(t+)$ or
$\alpha=P^\prime_\Phi(t-)$, then $E_\Phi(\alpha)\neq \emptyset$ and
$$
\htop(T,E_\Phi(\alpha))=\inf_{q>0}\{P_\Phi(q)-\alpha q\}=P_\Phi(t)-\alpha t,
$$

\item[(ii)] For $\alpha\in\bigcup_{t>0}
[P^\prime_\Phi(t-),P^\prime_\Phi(\infty)]$,
 $$
\inf_{q>0}\{P_\Phi(q)-\alpha q\}=\max\{h_\mu(T):\; \mu\in
\M(X,T),\;\Phi_*(\mu)=\alpha\}.
  $$

\item[(iii)] If~ $t>0$ such that $t\Phi$ has a unique
equilibrium state $\mu_t\in \mathcal{M}(X,T)$, then $\mu_t$ is
ergodic, $P^\prime_\Phi (t)=\Phi_*(\mu_t)$,
$E_{\Phi}(P^\prime_\Phi (t))\neq\emptyset$ and
$\htop(T,E_{\Phi}(P^\prime_\Phi (t)))=h_{\mu_t}(T)$.
\end{itemize}
\end{thm}

A significant part of Theorem \ref{thm-1.3}(i) is that we don't need the differentiability
assumption for $P_\Phi$. To the best of our knowledge, this result is not known even in the additive case.
It has a nice application in the multifractal analysis for certain probability measures on symbolic spaces
(see Remark \ref{re-4.10}).
We remark that the assumption of upper semi-continuity for  the entropy map is quite essential for the results in
Theorem \ref{thm-1.3}. This assumption is satisfied by some natural TDS's such as $h$-expansive TDS's \cite{Bow72} and more generally,  asymptotically
$h$-expansive TDS's \cite{Mis76} which include, for example,  $C^\infty$ transformations on Riemannian manifolds
\cite{Buz97}.
Without this assumption, the multifractal behavior may be very irregular and complicated.
See Section \ref{S-6} for some   examples. We remark that the differentiability property of the pressure functions
was studied in \cite{MaSm00, MaSm03} for rational maps on the Riemann sphere for certain additive potentials.

Meanwhile Theorems \ref{thm-1.1}-\ref{thm-1.3} are about asymptotically sub-additive potentials, our next theorem is concerned with  asymptotically additive potentials.  A TDS $(X,T)$ is called to be {\it saturated} if
for any $\mu\in \M(X,T)$, we have $G_\mu\neq \emptyset$ and $\htop(T,G_\mu)=h_\mu(T)$, where $G_\mu$ denotes the set of
{\it $\mu$-generic points} defined by
$$
G_\mu:=\left\{x\in X:\; \frac{1}{n}\sum_{j=0}^{n-1}\delta_{T^jx}\to \mu \mbox{ in the weak* topology as $n\to \infty$}
\right\},
$$
where $\delta_y$ denotes the probability measure whose support is the single point $y$.
It was shown independently in \cite{FLP08,PfSu07} that if a TDS $(X,T)$ satisfies the specification property,
then $(X,T)$ is saturated.

\begin{thm}
\label{thm-1.4}
Let $(X,T)$ be a TDS and let $\Phi$ be an asymptotically additive potential on $X$. Set
$\Omega=\{\Phi_*(\mu):\; \mu\in \M(X,T)\}$. Then $\Omega$ is a bounded closed interval. Furthermore we have the following statements.
\begin{itemize}
\item[(i)] $\{\alpha\in \R:\; E_\Phi(\alpha)\neq \emptyset\}\subseteq \Omega$.
\item[(ii)] If $\htop(T)<\infty$, then $P_\Phi$ is a real convex function over $\R$. Furthermore,
$$
\alpha\in \Omega \Longleftrightarrow \inf\{P_\Phi(q)-\alpha q:\; q\in \R\}\neq -\infty\Longleftrightarrow \inf\{P_\Phi(q)-\alpha q:\; q\in \R\}\geq 0.
$$
\item[(iii)] If  $\htop(T)<\infty$ and the entropy map is upper semi-continuous, then for each $\alpha\in \Omega$,
$$
\sup\{h_\mu(T):\; \mu\in \M(X,T),\; \Phi_*(\mu)=\alpha\}=\inf\{P_\Phi(q)-\alpha q:\; q\in \R\};
$$
Furthermore, for $\alpha\in \bigcup_{t\in \R}\{P^\prime_\Phi(t-),P^\prime_\Phi(t+)\}\cup P_\Phi^\prime(\pm\infty)$,
we have  $E_\Phi(\alpha)\neq \emptyset$ and $$
\htop(T,E_\Phi(\alpha))=\inf\{P_\Phi(q)-\alpha q:\; q\in \R\},
$$
where $P_\Phi^\prime(+\infty):=\lim_{q\to +\infty}P_\Phi(q)/q$ and $P_\Phi^\prime(-\infty):=\lim_{q\to -\infty}P_\Phi(q)/q$.
\item[(iv)] Assume that $(X,T)$ is saturated. Then $E_\Phi(\alpha)\neq \emptyset$ if and only if $\alpha\in \Omega$.
Furthermore, $\htop(T, E_\Phi(\alpha))=\sup\{h_\mu(T):\; \mu\in \M(X,T),\; \Phi_*(\mu)=\alpha\}$ for any $\alpha\in \Omega$.
\end{itemize}
\end{thm}

We remark that Theorem \ref{thm-1.4}(iii)-(iv) extend  previous results about  the Lyapunov spectrum  of continuous positive matrix-valued functions \cite{Fen03} and
the  Lyapunov spectrum  of certain asymptotically additive potentials \cite{KeSt04} on subshifts of finite type.

 In this paper, we also study the high dimensional Lyapunov spectrum. For  a finite family of asymptotically sub-additive
(resp. asymptotically additive)  potentials
$\Phi_i=\{\log \phi_{n,i}\}_{n=1}^\infty$, $i=1,\ldots,k$,  and $\a=(a_1,\ldots,a_k)$, we define
$$
E_{\bf \Phi}(\a)=\left\{x\in X:\;\lim_{n\to \infty}\frac{\log \phi_{n,i}(x)}{n}=a_i\mbox { for }1\leq i\leq k \right\}.
$$
We indeed obtain the high dimensional versions of Theorems \ref{thm-1.2}-\ref{thm-1.4} regarding the properties about
$\htop(T,E_{\bf \Phi}(\a))$ and the corresponding variational relations (see Theorems \ref{thm-1.2-h}, \ref{thm-1.3-h},
\ref{thm-6.1}).   For instance, when $(X,T)$ is a saturated TDS such that the entropy map is upper-semi continuous, then
for  any asymptotically additive potentials $\Phi_i$ ($i=1,\ldots,k$),  we  have
\begin{equation*}
\label{e-0404'}
E_{\bf \Phi}(\a)\neq \emptyset\Longleftrightarrow \a\in A:=\{{\bf \Phi}_*(\mu): \; \mu\in \M(X,T)\}
\end{equation*}
and
\begin{equation*}
\label{e-i1'}
\begin{split}
\htop(T,E_{\bf \Phi}(\a))&=\sup\{h_\mu(T):\; \mu\in \M(X,T) \mbox{ with } {\bf \Phi}_*(\mu)=\a\}\\
\mbox{}&=\inf\{P_{\bf \Phi}({\bf q})-\a\cdot {\bf  q}:\; {\bf q}\in \R^k\}, \qquad \forall\; \a\in A.
\end{split}
\end{equation*}
 where ${\bf \Phi}_*(\mu):=((\Phi_1)_*(\mu),\cdots,(\Phi_k)_*(\mu))$,
 $P_{\bf \Phi}({\bf q})=P(T,\sum_{i=1}^kq_i\Phi_i)$ for ${\bf q}=(q_1,\ldots,q_k)$, and $\a\cdot {\bf  q}$ denotes the inner product
 of $\a$ and ${\bf q}$ (see Theorem \ref{thm-6.1}(iii)-(iv)).

  As an application of the above result, we can improve a  result of Barreira and Gelfert in \cite{BaGe06} on Lyapunov exponents on nonconformal
 repellers. To see it, let $\Lambda$ be a repeller of a $C^{1}$ local diffeomorphism  $f:\;\R^2\to \R^2$,
  such that $f$ satisfies a cone condition on $\Lambda$ (see \cite{BaGe06} for the definition).
 Let $\Phi_i=\{\log \phi_{n,1}\}_{n=1}^\infty$ ($i=1,2$) are two potentials given by
 $$
 \phi_{n,i}(x)=\sigma_i(d_xf^n),\quad n\in \N,\; i=1,2,
 $$
 where $\sigma_i(d_xf^n)$ $(i=1,2$) denote the singular values of the differential of $f^n$ at $x$.
 Both $\Phi_1$ and $\Phi_2$ are asymptotically additive (see \cite[Proposition 4]{BaGe06}).
  Under the  additional assumptions  that $f$ is $C^{1+\delta}$ and $f$ has bounded distortion,  Barreira and Gelfert showed that
  $\htop(T,E_{\bf \Phi}(\a))=\inf\{P_{\bf \Phi}({\bf q})-\a\cdot {\bf  q}:\; {\bf q}\in \R^k\}$ for
  each gradient $\a$ of $P_{\bf \Phi}$ (see \cite[Theorem 1]{BaGe06}).
  However according to our result, these two additional assumptions can be removed (although in this case $P_{\bf \Phi}$ may be not differentiable)
  and the variational relation   holds for each $\a\in A:=\{{\bf \Phi}_*(\mu): \; \mu\in \M(X,T)\}$
  (we remark that $A$ contains the subdifferentials of $P_{\bf \Phi}$; see Theorem \ref{lem-mfsa1-1}). Below we give some further remarks.

\begin{re}
{\rm
\begin{itemize}
\item[(i)] In the definition of  sub-additive potential $\Phi=\{\log \phi_n\}$, we admit that $\phi_n(x)$ takes the value $0$. As an advantage, we can cover the interesting case that
$\phi_n(x)=\|\prod_{i=0}^{n-1}M(T^ix)\|$, where $M$ is an arbitrary continuous matrix-valued function.
\item[(ii)] There are some natural examples of  asymptotically sub-additive (resp. asymptotically  additive) potentials which may not be sub-additive, such as the general potential $$\Phi=\{\log \mu(I_n(x))\}_{n=1}^\infty,$$ where $\mu$ is a weak Gibbs measure on a full shift space over finite symbols and $I_n(x)$ denotes the $n$-th cylinder about $x$ (cf. \cite{FeOl03} and Proposition \ref{pro-le}(iv)).  By the way, the quotient space of all asymptotically additive potentials on $X$ under certain equivalence relation is a separable Banach space endowed with some norm (cf. Remark \ref{re-last}(ii)).      These are two main reasons that we setup the theory for  asymptotically sub-additive  potentials rather than sub-additive potentials.
\item[(iii)] For the proofs of Theorems \ref{thm-1.2}-\ref{thm-1.4}, we first prove their higher dimensional versions by applying convex analysis and the thermodynamic formalism, then derive the one-dimensional versions. Although it looks a bit strange and  there are relatively simple alternative approaches for the one-dimensional versions, however the extension to higher dimensions along those approaches seems difficult.
\item[(iv)] Let $\Phi=\{\log \phi_n\}_{n=1}^\infty$, where $\phi_n$'s are non-negative continuous functions on $X$ satisfying
$$
\phi_{n+m}(x)\leq C_n\phi_n(x)\phi_m(T^nx),\quad \forall\; n,m\in \N,\; x\in X,
$$
where $(C_n)$ is a sequence of positive numbers with $\lim_{n\to \infty}(1/n)\log C_n=0$. We do not know whether $\Phi$ is asymptotically sub-additive.
However, one can manage to prove Lemma \ref{lem-3.2} and Theorem \ref{thm-3.3} for this $\Phi$  by an approach similar to \cite{CFH08}. Furthermore, Theorems \ref{thm-1.1}-\ref{thm-1.3} remain valid for this kind of potentials.
\end{itemize}
}
\end{re}

 The content of the paper is following.
 In Section  \ref{S-2}, we give some preliminaries about topological entropy and
 topological pressures, and we also present and derive some results in  convex analysis that are needed in the proof of our theorems.
In section \ref{S-3}, we introduce the
 asymptotically sub-additive thermodynamic formalism and we also set up a formula for the
subdifferentials of pressure functions. The high dimensional
versions of Theorem \ref{thm-1.2}-\ref{thm-1.3} are formulated and
proved in Section \ref{S-4}.  In particular, we give a class of
sub-additive
 potentials on full shifts which satisfy part of (\ref{e-i1}) in Section \ref{S-4}.
 The high dimensional version of Theorem \ref{thm-1.4} is formulated and proved in Section \ref{S-5}.
 In Section \ref{S-6}, we give some examples about the irregular multifractal behaviors for additive potentials on
 TDS's for which the entropy maps is not upper semi-continuous.  In Appendix \ref{A},
 we give some properties about asymptotically sub-additive (resp. asymptotically additive potentials). In Appendix \ref{B}, we summarize the main notation and conventions
 used in this paper.

\section{Preliminaries}\label{S-2}
In this section, we first give  the definitions and some properties about
 the topological entropy of non-compact sets and  the topological pressure of non-additive potentials,
 for which the reader is referred to \cite{Bow73, Bar96, Pes-book, CFH08} for more details. Then we present some
 notation and known facts in convex analysis and derive several results which are need in the proofs of our main results.
\setcounter{equation}{0}

\subsection{Topological entropy}\label{S-2.1}
Let $(X, d)$ be a compact metric space and $T : X \to X$  a
continuous transformation. For any $n \in \N$ we define a new
metric $d_n$ on $X$ by
\begin{equation}
\label{e-2.1}
 d_n(x, y) =\max\left\{d\left(T^k(x), T^k(y)\right):\;k =0, \ldots,n-1 \right\},
 \end{equation}
 and for every $\epsilon>0$ we denote by $B_n(x, \epsilon)$ the open ball
of radius $\epsilon$ in the metric $d_n$ around $x$, i.e., $B_n(x,
\epsilon) =\{y \in X : \ d_n(x, y) < \epsilon\}$. Let $Z \subset X$ and $\epsilon> 0$. We say that an at
most countable  collection of balls $\Gamma
=\{B_{n_i}(x_i,\epsilon)\}_i$ covers $Z$ if $Z \subset \bigcup_i
B_{n_i}(x_i,\epsilon)$. For
$\Gamma=\{B_{n_i}(x_i,\epsilon)\}_i$, put $n(\Gamma) =\min_i
n_i$. Let $s \geq 0$ and define $$M(Z,s, N, \epsilon) =\inf\sum_i
\exp(-sn_i),$$
 where the infinum is taken over all collections $\Gamma=\{B_{n_i}(x_i,\epsilon)\}$ covering $Z$,
 such that $n(\Gamma) \geq  N$. The quantity $M(Z, s,
N, \epsilon)$ does not decrease with $N$, hence the following
limit exists: $$M(Z, s, \epsilon) = \lim_{N\to \infty} M(Z,s, N,
\epsilon).$$ There exists a critical value
of the parameter s, which we will denote by $\htop(T,Z,\epsilon)$,
where $M(Z, s, \epsilon)$ jumps from $\infty$ to $0$, i.e.
\[
M(Z, s, \epsilon) = \left\{
\begin{array}{ll}
0, & s > \htop(T,Z,\epsilon),\\
\infty,& s < \htop(T,Z,\epsilon).
\end{array}
\right.
\]
It is clear to see that $\htop(T,Z,\epsilon)$ does not decrease with $\epsilon$, and  hence the following limit exists, $$\htop(T, Z) = \lim_{\epsilon\to 0}\htop(T,Z,\epsilon).$$
 We  call
$\htop(T, Z)$ the {\it topological entropy of $T$ restricted to
$Z$} or, simply, the {\it topological entropy of $Z$}, when there
is no confusion about $T$. In particular we write $\htop(T)$ for $\htop(T,X)$. Here we recall some of the basic properties about the topological entropy.

\begin{pro}[\cite{Bow73,Pes-book}] \label{pro-2.1}
 The topological entropy as defined above satisfies the
following:
\begin{itemize}
\item[(1)] $\htop(T, Z_1) \leq\htop(T, Z_2)$ for any $Z_1
\subseteq Z_2\subseteq X$.
\item[(2)] $\htop(T, Z) =\sup_i\htop(T,
Z_i)$, where $Z =\bigcup_{i=1}^\infty Z_i\subseteq X$.
\item[(3)] Suppose $\mu$
is an invariant measure and $Z \subseteq X$ is such that $\mu(Z) =
1$, then $\htop(T, Z) \geq h_\mu(T)$, where $h_\mu(T )$ is the
measure-theoretic entropy.
\end{itemize}
\end{pro}

\subsection{Topological pressure}
\label{S-2.2}
Let $T:\ X\to X$ be a continuous transformation of a compact
metric space $(X,d)$.  For any $n\in \N$, define the metric $d_n$
as in (\ref{e-2.1}). For any $\epsilon>0$, a set $E\subseteq X$ is
said to be a {\it $(n,\epsilon)$-separated subset of $X$}  if
$d_n(x,y)>\epsilon$ for any two different points $x,y\in E$. Let
$\Phi=\{\log \phi_n(x)\}_{n=1}^\infty$ be a sequence of functions on $X$ for which $\phi_n$ is non-negative for each $n$.  We define
$$
P_n(T, \Phi,\epsilon) =\sup\left\{\sum_{x\in E}\phi_n(x):\; E \mbox{ is
a } (n,\epsilon) \mbox{-separated subset of } X\right\}.
$$
It is clear that $P_n(T, \Phi,\epsilon)$ is a decreasing function of
$\epsilon$. Define
$$
P(T,\Phi,\epsilon)=\limsup_{n\to \infty}\frac{1}{n}\log
P_n(T,\Phi,\epsilon)
$$
and $P(T,\Phi)=\lim_{\epsilon\to 0}P(T,\Phi,\epsilon)$. We call  $P(T,\Phi)$
 the {\it topological pressure of $\Phi$ with respect to
$T$} or, simply, the {\it topological pressure of $\Phi$}.

\subsection{Subdifferentials of convex functions}\label{S-2.3}
We first give some notation and basic facts in convex analysis. For details, one is referred to \cite{HiLe01, Roc70}.

 By a {\it convex combination} of points ${\bf x}_1$, $\ldots$, ${\bf
x}_m\in \mathbb{R}^k$ we mean a linear combination $\sum_{i=1}^m
\lambda_i{\bf x}_i$, where $\lambda_1+\cdots+\lambda_m=1$ and
$\lambda_1,\ldots,\lambda_m\ge 0$. For any subset $M$ of
$\R^k$, the {\it convex hull}  ${\rm conv}(M)$ of $M$ is the set of all convex
combinations of points from $M$. Carath\'{e}odory's Theorem says that for any
subset $M$ of $\mathbb{R}^k$, the convex hull ${\rm conv}(M)$ is the set
of all convex combinations of $k+1$ points from $M$ (cf. \cite[Theorem 17.1]{Roc70}).

Let $C$ be a convex subset of $\mathbb{R}^k$. A point ${\bf x}\in C$
is called an {\it extreme point} of $C$ if ${\bf x}=p{\bf y}+(1-p){\bf z}$ for some ${\bf y}, {\bf z}\in
C$ and $0<p<1$, then ${\bf
x}={\bf y}={\bf z}$. The set of extreme points of $C$ is denoted by
${\rm ext}(C)$.  Minkowski's Theorem says that for any  non-empty compact
convex subset $C$ of $\mathbb{R}^k$,  $C={\rm conv}({\rm ext}(C))$ (cf. \cite[Theorem 2.3.4]{HiLe01} or
\cite[Corollary 18.5.1]{Roc70}). Hence, according to Carath\'{e}odory's Theorem and Minkowski's Theorem,  each
point in a compact convex set $C\subset \R^k$ is a convex combination of $k+1$ points from ${\rm ext}(C)$.

A point ${\bf x}\in C$ is called {\it exposed  point} of $C$, if
$\{\bf x\}$ is the intersection of $C$ with some supporting
hyperplane of $C$. The set of all exposed points of $C$ will be
denoted by ${\rm expo}(C)$. Straszewicz' Theorem says for  any compact convex set $C$ in $\mathbb{R}^k$,
 ${\rm expo}(C)$ is a dense subset of
${\rm ext}(C)$ (cf. \cite[Theorem 18.6]{Roc70}).

Let $U$ be an open convex subset of $\mathbb{R}^k$ and $f$ be a real
continuous convex function on $U$. For ${\bf x}\in U$,
${\bf a}\in \mathbb{R}^k$ is called a {\it subgradient}
of $f$ at  ${\bf x}$, if for any ${\bf y}\in U$ one has
$$f({\bf y})-f({\bf x})\ge {\bf a} \cdot ({\bf y}-{\bf x}),$$
where the dot denotes the dot product. The set of all subgradients
at ${\bf x}$ is called the {\it subdifferential} of $f$ at ${\bf x}$
and is denoted $\partial f({\bf x})$. For ${\bf x}\in U$, the subdifferential $\partial f({\bf x})$ is always
a nonempty convex compact set (cf. \cite[Theorem 23.4]{Roc70}). Write
$$\partial^e f({\bf x})={\rm ext}(\partial f({\bf x})).$$
When $\partial^e f({\bf x})=\{\bf a\}$, we say that $f$ is
{\it differentiable} at ${\bf x}$ and write $f'({\bf
x})=\a$. It is known that  $f$ is differentiable
for almost every ${\bf x}\in U$ (cf. \cite[Theorem 4.2.3]{HiLe01}). In the case $k=1$,  $\partial f( x)=[f'(x-),f'(x+)]$ and $\partial^e f( x)=\{f'(x-),f'(x+)\}$.

Next we define
\begin{equation}
\label{e-0330}
\partial f(U)=\bigcup_{{\bf x}\in U} \partial f({\bf x}) \quad \text{ and } \quad
\partial^e f(U)=\bigcup_{{\bf x}\in U} \partial^e f({\bf x}).
\end{equation}

\begin{pro}\label{lem-key2.2} Let $U$ be an open convex subset of $\mathbb{R}^k$
and $f$ be a real continuous convex function on $U$. Then for each
 ${\bf x}\in U$ and $\a\in \partial^e f({\bf x})$, there
exists a sequence $({\bf x}_n)\subset U$  such that  $\lim
\limits_{n\rightarrow \infty} {\bf x}_n={\bf x}$,  $f$ is
differentiable at each point ${\bf x}_n$ and $\a=\lim
\limits_{n\rightarrow \infty} f'({\bf x}_n)$.
\end{pro}
\begin{proof} Let ${\bf x}\in U$. Since ${\rm expo}(\partial f(x))$ is
dense in $\partial ^ef(x)$, we only need to show that the lemma holds when $\a \in {\rm expo}(\partial f({\bf x}))$. Fix  such an $\a$ and write  $\a=(a_1,\cdots, a_k)$. Then there exists a non-zero vector ${\bf
t}=(t_1,\cdots,t_k)\in \mathbb{R}^k$ such that
\begin{equation}\label{eq-unique}
{\bf t}\cdot {\bf b}<{\bf t}\cdot \a
\quad \text{ for any } {\bf b}\in \partial f({\bf x})\setminus
\{\a\}.
\end{equation}

Since $f$ is differentiable almost every on $U$, there
exists a sequence $({\bf x}_n)\in U$  such that $\lim
\limits_{n\rightarrow \infty} {\bf x}_n={\bf x}$, $f$ is
differentiable at each ${\bf x}_n$ and
\begin{equation}\label{e-2.2u} |{\bf x}_n-({\bf x}+{\bf
t}/n)|<n^{-2}\quad \mbox{for all }  n\in \mathbb{N}.
\end{equation}
 Write
${\bf a}_n=f'({\bf x}_n)$. Note that the sequence
$({\bf a}_n)$ is bounded because of the boundedness of $(x_n)$. Hence by taking a subsequence if it is necessary, we can assume that  $\lim \limits_{n\rightarrow
\infty}{\bf a}_n=\a'$ for some
$\a'\in \mathbb{R}^k$. In the following we show that $\a'=\a$.

Since ${\bf a}_n=f'({\bf x}_n)$,
 one has
$$
f({\bf z})-f({\bf x}_n)\ge {\bf a}_n\cdot ({\bf z}-{\bf x}_n)\quad \mbox{ for any } {\bf z}\in U.
$$
 Letting $n\rightarrow \infty$ yields
$f({\bf z})-f({\bf x})\ge \a'\cdot ({\bf z}-{\bf x})$ for any ${\bf z}\in U$,
which implies $\a'\in \partial f({\bf x})$.
Meanwhile for each $n\in \N$,
$$f({\bf x})-f({\bf x}_n)\ge {\bf a}_n \cdot ({\bf x}-{\bf x}_n)
\quad \text{ and }\quad f({\bf x}_n)-f({\bf x})\ge \a\cdot ({\bf
x}_n-{\bf x}).
$$
Hence $\a_n\cdot ({\bf x}_n-{\bf x})\ge
f({\bf x}_n)-f({\bf x})\ge \a\cdot ({\bf x}_n-{\bf
x})$.  That is, $$\a_n\cdot ({\bf t}+n{\bf
w}_n)\ge \a\cdot ({\bf t}+n{\bf w}_n),$$
where
${\bf w}_n:={\bf x}_{n}-({\bf
x}+{\bf t}/{n})$. Taking $n\to \infty$ and noting that
  $\lim \limits_{n\rightarrow \infty} n|{\bf
w}_n|=0$ (by (\ref{e-2.2u})), we have
$$\a'\cdot {\bf t}\ge \a \cdot {\bf t}.$$
Combining  it with \eqref{eq-unique} and the fact
$\a'\in \partial f({\bf x})$, one has
$\a'=\a$. This finishes the proof of the proposition.
\end{proof}

\begin{pro}
\label{pro-2.3}
Let $Y$ be a compact convex subset of a topological vector space which satisfies the first axiom of countability
(i.e., there is a countable base at each point)  and $U\subseteq \R^k$ a non-empty open convex set. Suppose $f:\; U\times Y\to \R\cup\{-\infty\}$
is a map satisfying the following  conditions:
\begin{itemize}
\item[(i)] $f(\bq,y)$ is convex in $\bq$;
\item[(ii)] $f(\bq,y)$ is affine in $y$;
\item[(iii)] $f$ is upper semi-continuous over $U\times Y$;
\item[(iv)] $g(\bq):=\sup_{y\in Y}f(\bq,y)>-\infty$ for any $\bq\in U$.
\end{itemize}
For each $\bq\in U$, denote $
\I(\bq):=\{y\in Y:\; f(\bq,y)=g(\bq)\}.
$
Then
\begin{equation}
\label{e-1504}
\partial g(\bq)=\bigcup_{y\in \I(\bq)}\partial f(\bq,y),
\end{equation}
where $\partial f(\bq,y)$ denotes the subdifferential of $f(\cdot,y)$ at $\bq$.
\end{pro}
\begin{proof}
By (i)-(iv), $g$ is a real convex function over $U$, and $\I(\bq)$ is a non-empty compact convex subset of $Y$
for each $\bq\in U$.  For convenience, denote by $R(\bq)$ the righthand side of (\ref{e-1504}). A direct check shows
that $R(\bq)$ is a non-empty convex subset of $\R^k$ for each $\bq\in U$. We further show that for each $\bq\in U$,
\begin{itemize}
\item[(c1)] $R(\bq)$ is compact;
\item[(c2)] For each $\delta>0$, there exists $\gamma>0$ such that
$$
R(\bt)\subseteq B_\delta(R(\bq))
\quad \mbox{ whenever }\bt \in U,\; |\bt-\bq|<\gamma,
$$
where  $B_\delta(R(\bq)):=\{{\bf b}\in \R^k:\; d({\bf b},R(\bq))\leq \delta\}$ is the closed $\delta$-neighborhood of $R(\bq)$ in $\R^k$.
\end{itemize}
To show (c1), let $(\a_n)$ be a sequence in $R(\bq)$. Take $y_n\in \I(\bq)$ so that $\a_n\in \partial f(\bq,y_n)$.  Then
\begin{equation}\label{e-1504a}
f(\bt,y_n)-g(\bq)=f(\bt,y_n)-f(\bq,y_n)\geq \a_n\cdot (\bt-\bq)
\end{equation}
for each $\bt\in U$. Hence the sequence $(\a_n)$ should be bounded (otherwise, there exists $\bt\in U$ such that $\a_n\cdot (\bt-\bq)$ is unbounded from above,
however $f(\bt,y_n)-f(\bq,y_n)=f(\bt,y_n)-g(\bq)\leq g(\bt)-g(\bq)$).   Taking a subsequence if necessary, we assume
that $y_n\to y$ and $\a_n\to \a$ for some $y\in \I(\bq)$ and $\a\in \R^k$.
Since $f(\bt,\cdot)$ is upper semi-continuous, by (\ref{e-1504a}) we have
$f(\bt,y)-f(\bq,y)=f(\bt,y)-g(\bq)\geq \a\cdot(\bt-\bq)$ for each $\bt \in U$. This shows $\a\in R(\bq)$
and hence $R(\bq)$ is compact. To show (c2), we use contradiction. Assume that (c2) does not hold.
Then there exist $\delta>0$ and a sequence
$(\bt_n)$ in $U$ with $\lim_{n\to \infty}\bt_n=\bq$ such that there exists $\a_n\in R(\bt_n)$ satisfying $d(\a_n,R(\bq))>\delta$
for each $n$. Take $y_n\in \I(\bt_n)$ so that $\a_n\in \partial f(\bt_n,y_n)$.  Then we have
\begin{equation}\label{e-1504b}
f(\bt,y_n)-g(\bt_n)=f(\bt,y_n)-f(\bt_n,y_n)\geq \a_n\cdot (\bt-\bt_n)
\end{equation}
for each $\bt\in U$. Similarly we can show that $\a_n$ is bounded. Hence by taking a subsequence if necessary, we can assume
that $y_n\to y$ and $\a_n\to \a$ for some $y\in Y$ and $\a\in \R^k$. By the upper semi-continuity of $f$ and
the continuity of $g$, we have
$$
f(\bq,y)\geq \limsup_{n\to \infty}f(\bt_n,y_n)=\limsup_{n\to \infty} g(\bt_n)=g(\bq),
$$
which implies $y\in \I(\bq)$. Hence taking  $n\to \infty$ in (\ref{e-1504b}) yields
 $$
 f(\bt,y)-f(\bq,y)=f(\bt,y)-g(\bq)\geq \limsup_{n\to \infty} (f(\bt,y_n)-g(\bt_n))\geq \a\cdot (\bt-\bq)
 $$
 for each $\bt\in U$. Hence $\a\in \partial f(\bq,y)$. Thus $\a\in R(\bq)$, which contradicts the assumption that
 $d(\a_n,R(\bq))>\delta$ for all $n$. This finishes the proof of (c2).

  Now we are ready to show (\ref{e-1504}), i.e., $\partial g(\bq)=R(\bq)$. For each $\a\in R(\bq)$,
  there exists $y\in \I(\bq)$ so that $\a\in \partial f(\bq,y)$. Hence for each $\bt\in U$,
\begin{equation}
\label{e-1504c}
 g(\bt)-g(\bq)\geq  f(\bt,y)-g(\bq)=f(\bt,y)-f(\bq,y)\geq \a\cdot (\bt-\bq).
  \end{equation}
This implies $\a\in \partial g(\bq)$ and thus $\partial g(\bq)\supseteq R(\bq)$.

In the end, we show that $\partial g(\bq)\subseteq R(\bq)$ by contradiction.
Assume that $\a\in \partial g(\bq)\setminus R(\bq)$.
Since $R(\bt)\subseteq \partial g(\bt)$, we have
\begin{equation}
\label{e-1504d}
g(\bt)-g(\bq)\leq \b\cdot (\bt-\bq),\qquad \forall \; \bt\in U,\; \b\in R(\bt).
\end{equation}
Note that  $\a\in \partial g(\bq)$. We have $g(\bt)-g(\bq)\geq \a\cdot (\bt-\bq)$ for all $\bt\in U$. This combining (\ref{e-1504d}) yields
\begin{equation}
\label{e-1504e}
\a\cdot (\bt-\bq)\leq \b\cdot (\bt-\bq),\qquad \forall \; \bt\in U,\; \b\in R(\bt).
\end{equation}
Since $\a\not\in R(\bq)$ and $R(\bq)$ is compact, there exists $\delta>0$ so that $\a\not\in B_\delta(R(\bq))$. Notice that
 $B_\delta(R(\bq))$ is compact convex (since so is  $R(\bq)$), there exists a vector ${\bf e}\in \R^k$ such that $|{\bf e}|=1$ and
 $\a\cdot {\bf e}>\b\cdot {\bf e}$ for any $\b\in B_\delta(R(\bq))$. By (c2), there exists $\gamma>0$ such that
 $R(\bt)\subseteq B_\delta(R(\bq))$ whenever $|\bt-\bq|\leq \gamma$.    Take a small $0<\tilde{\gamma}<\gamma$ such that
 $\bt_0:=\bq+(\tilde{\gamma}/2){\bf e}\in U$. Then $\a\cdot (\bt_0-\bq)>\b\cdot (\bt_0-\bq)$ for any $\b\in R(\bt_0)$, which
 contradicts (\ref{e-1504e}). This proves $\partial g(\bq)\subseteq R(\bq)$.
\end{proof}
\subsection{Conjugates of convex functions}
\label{S-2.4}

Let $f:\R^k\to \R\cup\{+\infty\}$ be convex and not identically equal to $+\infty$. Then the function $f^*:\; \R^k\to \R\cup\{+\infty\}$ defined by
$$
{\bf s}\mapsto f^*({\bf s}):=\sup\{{\bf s}\cdot {\bf x}-f({\bf x}):\; x\in \R^k\}
$$
is called the {\it conjugate function} of $f$ or {\it Legendre
transform} of $f$. It is known that $f^*$ is also convex and not
identically equal to $+\infty$ (cf. \cite[p. 211]{HiLe01}). Let
$f^{**}$ denote the conjugate of $f^*$. The following  result is
well known (cf. \cite[Theorem 12.2]{Roc70}).
\begin{thm}
\label{thm-R}
Let $f:\R^k\to \R\cup\{+\infty\}$ be convex and not identically equal to $+\infty$. Let ${\bf x}\in \R^k$. Assume that $f$ is lower semi-continuous
at ${\bf x}$, i.e., $\liminf_{{\bf y}\to {\bf x}}f({\bf y})\geq f({\bf x})$.  Then $f^{**}({\bf x})= f({\bf x})$.
\end{thm}

As an application, we have

\begin{cor}
\label{cor-0330}
Let $A$ be a non-empty convex set in $\R^k$ and $g:\; A\to \R$ be a concave function.
Set $$
W({\bf x})=\sup\{g(\a)+\a\cdot {\bf x}:\; \a\in A\},\qquad {\bf x}\in \R^k
$$
and
$$
G(\a)=\inf\{W({\bf x})-\a\cdot {\bf x}:\; {\bf x}\in \R^k\},\qquad \a\in \R^k.
$$
Then we have
 \begin{itemize}
\item[(i)]
$G(\a)=g(\a)$ for $\a\in {\rm ri}(A)$, where ${\rm ri}(A)$ denotes the relative interior of $A$.
\item[(ii)]
Assume in addition that  $A$ is closed. If $g$ is upper semi-continuous at $\a\in A$, then $G(\a)=g(\a)$.
\end{itemize}
\end{cor}
\begin{proof}
Let $f:\; \R^k\to \R\cup\{+\infty\}$ be the function which agrees with $-g$ on $A$ but is $+\infty$ everywhere else.
Then $f$ is convex and has $A$ as its efficient domain, i.e., $A=\{{\bf x}:\; f({\bf x})<+\infty\}$.
 By the definition of $W$ and $G$, we have $W=f^*$ and $G=-f^{**}$. However,  $f$ is lower semi-continuous on ${\rm ri}(A)$
(see, e.g., \cite[Theorem 7.4]{Roc70}). Hence by
Theorem \ref{thm-R}, we have $f^{**}(\a)=f(\a)$ for $\a\in {\rm ri}(A)$, and thus
$G(\a)=g(\a)$ for $\a\in {\rm ri}(A)$. This proves (i). To show (ii), assume that $A$ is closed. Let $\a\in A$ so that
$g$ is upper semi-continuous at $\a$. Then it is direct to check that $f$ is lower semi-continuous at $\a$.
By Theorem \ref{thm-R}, we have $f^{**}(\a)=f(\a)$ and  hence
$G(\a)=g(\a)$. This finishes the proof of (ii).
\end{proof}

\section{The thermodynamic formalism and subdifferentials of pressure functions}
\label{S-3}

\setcounter{equation}{0}

 In this section, we firstly introduce a variational
principle of topological pressures which plays a key role in the
proofs of our main theorems. Then we set up a formula for the
subdifferentials of pressure functions.

Let $(X,T)$ be a TDS and let $\Phi=\{\log \phi_n\}_{n=1}^\infty$ be
an asymptotically sub-additive potential on a TDS $(X,T)$. Let
$\lambda_\Phi$, $\Phi_*$ and $\overline{\beta}(\Phi)$ be defined as
in (\ref{e-1.1}), (\ref{e-1.2}) and (\ref{e-1.5}). Some basic
properties of $\lambda_\Phi$, $\Phi_*$ and $\overline{\beta}(\Phi)$
are given in Appendix \ref{A}. The following variational principle
plays a key role in our analysis.
\begin{thm}[\cite{CFH08}]
\label{thm-3.3}
 The topological pressure $P(T,\Phi)$ of $\Phi$
satisfies the following variational principle:
\begin{eqnarray*}
P(T,\Phi)=\left\{\begin{array}{l} -\infty, \qquad\mbox{ if } \Phi_*(\mu)=-\infty \mbox{ for all }\mu\in \M(X,T),\\
\sup\{h_\mu(T)+\Phi_*(\mu):\; \mu\in \M(X,T), \; \Phi_*(\mu)\neq
-\infty\},\mbox{ otherwise}.
\end{array}
\right.
\end{eqnarray*}
In particular if $\htop(T)<\infty$, then
$P(T,\Phi)=\sup\{h_\mu(T)+\Phi_*(\mu):\ \mu\in \M(X,T)\}$.
\end{thm}

The above theorem is only proved in \cite[Theorem 1.1]{CFH08} for
sub-additive potentials. However the  proof given there works well
for  asymptotically sub-additive potentials, in which  we only need
to replace Lemma 2.3 in \cite{CFH08} by Lemma \ref{lem-3.2} given in
Appendix. We remark that the  variational principle for sub-additive
potentials has been studied in \cite{Fal88, Bar96, Fen04, Kae04,
KeSt04, Bar06, Mum06} under additional assumptions on the
corresponding sub-additive potential and TDS.

In the remain part of this section, we present and prove a formula
for the subdifferentials of pressure functions. We first give some
notation.

Let $k\in \N$. For each $i=1,\ldots,k$, let $\Phi_i=\{\log
\phi_{n,i}\}_{n=1}^\infty$ be an asymptotically sub-additive
potential on $(X,T)$. Write $\mathbb{R}^k_+=\{
(x_1,x_2,\ldots,x_k):x_i>0,\;i=1,2,\cdots,k\}$ and ${\bf
\Phi}=(\Phi_1,\Phi_2,\ldots,\Phi_k)$. For $\mu\in \M(X,T)$, write
$$
{\bf \Phi}_*(\mu)=((\Phi_1)_*(\mu),\cdots, (\Phi_k)_*(\mu)).
$$

For ${\bf q}=(q_1,\cdots,q_k)\in \mathbb{R}^k_+$, let ${\bf q}\cdot
{\bf \Phi}=\sum_{i=1}^k q_i\Phi_i$ denote the asymptotically
sub-additive potential $\{\sum_{i=1}^kq_i\log
\phi_{n,i}\}_{n=1}^\infty$ and write
\begin{equation}\label{e-2.2}
P_{{\bf \Phi}}({\bf q})=P(T,{\bf q}\cdot {\bf \Phi}).
\end{equation}
We call  $P_{{\bf \Phi}}$  the {\it pressure function} of ${\bf
\Phi}$.

Let $\overline{\beta}({\bf \Phi})=\overline{\beta}(\sum_{i=1}^k
\Phi_i)$. Then by Theorem \ref{thm-3.3}, if $\overline{\beta}({\bf
\Phi})=-\infty$ then $P_{{\bf \Phi}}({\bf q})=-\infty$ for any ${\bf
q}\in \mathbb{R}^k_+$. If $\overline{\beta}({\bf \Phi})>-\infty$,
then $\overline{\beta}(\Phi_1)>-\infty$, $\ldots$,
$\overline{\beta}(\Phi_k)>-\infty$.

\begin{pro} \label{pro-mfsa-1} Assume that $\htop(T)<\infty$ and $\overline{\beta}({\bf \Phi})
>-\infty$. Then  $P_{{\bf \Phi}}$ is a real continuous  convex
function on $\mathbb{R}^k_+$ and
$$\partial
P_{{\bf \Phi}}(\mathbb{R}^k_+)\subseteq
(-\infty,\overline{\beta}(\Phi_1)]\times
(-\infty,\overline{\beta}(\Phi_2)]\times \cdots \times
(-\infty,\overline{\beta}(\Phi_k)],$$ where $\partial P_{{\bf
\Phi}}(\mathbb{R}^k_+)$ is defined as in (\ref{e-0330}).
\end{pro}
\begin{proof} By Lemma \ref{lem-mf-sa}(4),
 $P_{{\bf \Phi}}({\bf q})\in \mathbb{R}$
for ${\bf q}\in \mathbb{R}^k_+$. The convexity of $P_{{\bf \Phi}}$
over $\mathbb{R}^k_+$ just comes from Theorem \ref{thm-3.3}, using
the affine property of the maps $\mu\to h_\mu(T)$ and $\mu\to
(\Phi_i)_*(\mu)$. Since $P_{{\bf \Phi}}$ is also locally bounded on
$\mathbb{R}^k_+$, $P_{{\bf \Phi}}$ is continuous on
$\mathbb{R}^k_+$.

Fix ${\bf q}=(q_1,\cdots,q_k)\in \mathbb{R}^k_+$. Define  ${\bf
q}_\lambda=(q_1+\lambda,q_2,\cdots,q_k)$ for $\lambda>0$. Let
$\a=(a_1,\cdots,a_k)\in \partial P_{\bf \Phi}$. Then
$$h_{\text{top}}(T)+\lambda \overline{\beta}(\Phi_1)+\sum_{i=1}^k q_i
 \overline{\beta}(\Phi_i)\ge P_{{\bf \Phi}}({\bf
q}_\lambda)\ge P_{{\bf \Phi}}({\bf q})+({\bf q}_\lambda-{\bf
q})\cdot \a=P_{\bf \Phi}({\bf q})+ \lambda a_1.
$$
Letting $\lambda\rightarrow \infty$ one gets
$\overline{\beta}(\Phi_1)\ge \alpha_1$. Similarly, we have
$\alpha_i\le \overline{\beta}(\Phi_i)$ for $i=2,\cdots,k$.
\end{proof}

For ${\bf q}\in \mathbb{R}^k_+$, let $\mathcal{I}({\bf \Phi},{\bf
q})$ denote the collection of invariant measures $\mu$ such that
$$h_\mu(T)+ {\bf q}\cdot {\bf \Phi}_*(\mu)=P(T,{\bf q}\cdot {\bf \Phi}).$$ If
$\mathcal{I}({\bf \Phi},{\bf q})\neq \emptyset$, then each element
$\mathcal{I}({\bf \Phi},{\bf q})$ is called an {\it equilibrium
state} for ${\bf q}\cdot {\bf \Phi}$.

 In the following theorem, we
set up a formula for the subdifferentials of $P_{{\bf \Phi}}$, which
extends  Ruelle's derivative formula
 for the pressures of  additive potentials (cf. \cite[exercise 5,
p. 99]{Rue-book}, \cite[lemma 4]{Oli99} and \cite[theorem
4.3.5]{Kel-book}).

\begin{thm} \label{lem-mfsa1-1} Assume that
$\htop(T)<\infty$, $\overline{\beta}({\bf \Phi})>-\infty$, and that
the entropy map $\mu\mapsto h_\mu(T)$ is upper semi-continuous. Then
\begin{itemize}
\item[(i)]For
any ${\bf q}\in \mathbb{R}^k_+$, $\mathcal{I}({\bf \Phi},{\bf q})$
is a non-empty compact convex subset of $\mathcal{M}(X,T)$, and
every extreme point of $\mathcal{I}({\bf \Phi},{\bf q})$ is an
ergodic measure {\rm (}i.e., an extreme point of
$\mathcal{M}(X,T)${\rm )}. Furthermore
\begin{equation}\label{eq-de-1}
\partial P_{{\bf \Phi}}({\bf
q})=\{{\bf \Phi}_*(\mu):\mu\in \mathcal{I}({\bf \Phi},{\bf q})\}.
\end{equation}
\item[(ii)] Assume in addition that  $\Phi_i$ ($i=1,\ldots,k$) are all asymptotically additive. Then the above results hold for all $\bq \in \R^k$.
\end{itemize}
\end{thm}
\begin{proof}To prove (i), let ${\bf q}\in \mathbb{R}^k_+$. Then $\bq\cdot {\bf \Phi}$ is an asymptotically sub-additive potential.
 We first show that $\mathcal{I}({\bf \Phi},{\bf q})
\neq \emptyset$. By Theorem \ref{thm-3.3}, there exists a sequence
$\{\mu_n\}\subset \mathcal{M}(X,T)$ such that $P_{{\bf \Phi}}({\bf
q})=\lim_{n\to \infty}h_{\mu_n}(T)+{\bf q}\cdot {\bf
\Phi}_*(\mu_n)$. Let $\mu$ be a limit point of $\{\mu_n\}$. Then by
the upper semi-continuity of $h_{(\cdot)}(T)$ and
$(\Phi_i)_*(\cdot)$, we have $P_{{\bf \Phi}}({\bf q})\leq
h_\mu(T)+{\bf q}\cdot {\bf \Phi}_*(\mu)$. Applying  Theorem
\ref{thm-3.3} again we obtain $P_{{\bf \Phi}}({\bf q})=
h_\mu(T)+{\bf q}\cdot {\bf \Phi}_*(\mu)$, i.e., $\mu\in
\mathcal{I}({\bf \Phi},{\bf q})$. Hence $\mathcal{I}({\bf \Phi},{\bf
q}) \neq \emptyset$.

An  identical argument  shows that any limit point of
$\mathcal{I}({\bf \Phi},{\bf q})$ belongs to $\mathcal{I}({\bf
\Phi},{\bf q})$ itself. Therefore $\mathcal{I}({\bf \Phi},{\bf q})$
is closed and thus compact. Now assume that $\mu$ is an extreme
point of $\mathcal{I}({\bf \Phi},{\bf q})$. We claim that  $\mu$ is
ergodic, i.e., $\mu$ is also an extreme point of $\mathcal{M}(X,T)$.
To see it, assume $\mu=p\mu_1+(1-p)\mu_2$ for some $p>0$ and
$\mu_1,\mu_2\in \mathcal{M}(X,T)$. Since
$h_\mu(T)=ph_{\mu_1}(T)+(1-p)h_{\mu_2}(T)$ and ${\bf
\Phi}_*(\mu)=p{\bf \Phi}_*(\mu_1) +(1-p){\bf \Phi}_*(\mu_2)$, we
have
\begin{eqnarray*}
P_{{\bf \Phi}}({\bf q})&=&
h_\mu(T)+{\bf q}\cdot {\bf \Phi}_*(\mu)\\
&=&p\left(h_{\mu_1}(T)+{\bf q}\cdot {\bf \Phi}_*(\mu_1)\right)
+(1-p)\left(h_{\mu_2}(T)+{\bf q}\cdot {\bf \Phi}_*(\mu_2)\right).
\end{eqnarray*}
By Theorem \ref{thm-3.3}, we have  $\mu_1,\mu_2\in \mathcal{I}({\bf
\Phi},{\bf q})$. Since $\mu$ is an extreme point of $\I({\bf
\Phi},{\bf q})$, we have $\mu_1=\mu_2=\mu$. This shows that $\mu$ is
also an extreme point of $\mathcal{M}(X,T)$.

Next we show \eqref{eq-de-1}. In Proposition \ref{pro-2.3}, we take
$Y=\M(X,T)$, $U=\R^k_+$. Define $f:U\times Y\to \R\cup\{-\infty\}$
by
$$f(\bq,\mu)=\bq\cdot {\bf \Phi}_*(\mu)+h_\mu(T).$$
Then $f$ satisfies the conditions (i)-(iv) in Proposition
\ref{pro-2.3}. The identity  \eqref{eq-de-1} just comes from
 (\ref{e-1504}). This finishes the proof of (i).

 Now we turn to the proof of (ii). Assume  $\Phi_i$ ($i=1,\ldots,k$) are all asymptotically additive. Let $\bq\in \R^k$.
 Then $\bq\cdot {\bf \Phi}$ is also asymptotically additive. Clearly the above proof still work for this case (as a slightly different point, we should take $U=\R^k$ for
 the proof of \eqref{eq-de-1}).
\end{proof}

\section{multifractal formalism for aymptotically sub-additive potentials }
\label{S-4}
\setcounter{equation}{0}

In the section, we establish the multifractal formalism for
asymptotically sub-additive potentials. Let $(X,T)$ be a TDS.

\subsection{Proof of Theorem \ref{thm-1.1}}
Let $\Phi=\{\log \phi_n\}_{n=1}^\infty$ be an asymptotically
sub-additive potential on $(X,T)$ with
$\overline{\beta}(\Phi)>-\infty$.  For $x\in X$, we denote by $V(x)$
the set of all limit points  in $\M(X)$ of the sequence
$\mu_{x,n}=(1/n)\sum_{j=0}^{n-1}\delta_{T^jx}$. This set is a
non-empty compact subset of $\M(X,T)$ for each $x$ (cf.
\cite{Bow73}). The following result of Bowen plays a key role in the
proof of Theorem \ref{thm-1.1}.

\begin{lem}[Bowen \cite{Bow73}]
\label{Bow} For $t\geq 0$, define
$$
R(t)=\{x\in X:\; \exists\; \mu\in V(x) \mbox{ with }h_\mu(T)\leq
t\}.
$$
Then $\htop (T, R(t))\leq t$.
\end{lem}

\begin{proof}[Proof of Theorem \ref{thm-1.1}]
Let $\alpha=\overline{\beta}(\Phi)$. By Lemma \ref{lem-mf-sa}(2),
there exists  $\mu\in \E(X,T)$ so that $\Phi_*(\mu)=\alpha$.
By Proposition \ref{pro-1195}(1), $\mu(E_{\Phi}(\alpha))=1$.
Thus $E_{\Phi}(\alpha)\neq \emptyset$. Furthermore by
Proposition \ref{pro-2.1}(3), $h_{\text{top}}(T,
E_{\Phi}(\alpha))\ge h_\mu(T)$. This indeed proves
\begin{equation}
\label{e-thm1} \htop(T,E_{\Phi}(\alpha))\geq \sup \{
h_\mu(T): \mu\in \mathcal{E}(X,T),\; \Phi_*(\mu)=\alpha\}.
\end{equation}

Now assume  $\nu\in \mathcal{M}(X,T)$ so that
$\Phi_*(\nu)=\alpha$. Let $\nu=\int_{\mathcal{E}(X,T)} \theta
~d m(\theta)$ be the ergodic decomposition of $\mu$. By Proposition
\ref{pro-1195}(3),
$\alpha=\Phi_*(\nu)=\int_{\mathcal{E}(X,T)}
\Phi_*(\theta) ~d m(\theta)$. Since $\alpha\geq \Phi_*(\theta)$
for each $\theta\in \E(X,T)$, we have $\alpha=\Phi_*(\theta)$ whenever
$\theta\in \Omega'$, where $\Omega'$ is a subset of $\E(X,T)$ with
$m(\Omega')=1$. Hence
\begin{align*}
h_\nu(T)+\Phi_*(\nu)&=\int_{\mathcal{E}(X,T)} \left (
h_{\theta}(T)+\Phi_*(\theta) \right ) ~d
m(\theta)=\int_{\Omega'} \left ( h_{\theta}(T)+\Phi_*(\theta)
\right ) ~d m(\theta)\\
&\le \int_{\Omega'} \sup \{ h_\mu(T):\mu\in \mathcal{E}(X,T),
\Phi_*(\mu)=\alpha\} ~ d m(\theta)+\alpha\\
&=\sup \{ h_\mu(T):\mu\in \mathcal{E}(X,T),
\Phi_*(\mu)=\alpha\}+\alpha.
\end{align*}
This proves
\begin{equation}\label{e-thm1-2}
\sup\{h_\mu(T): \mu\in \mathcal{M}(X,T),\;
\Phi_*(\mu)=\alpha\}\leq \sup\{h_\mu(T): \mu\in
\mathcal{E}(X,T),\; \Phi_*(\mu)=\alpha\}.
\end{equation}

Next we prove  that
\begin{equation}
\label{e-thm1-3} \htop(T,E_{\Phi}(\alpha))\le \sup \{
h_\mu(T): \mu\in \mathcal{M}(X,T),\; \Phi_*(\mu)=\alpha\}.
\end{equation}
Denote by $t$ the right-hand side of the above inequality.  We may
assume that $t<\infty$, otherwise there is nothing remained to
prove.
 Let $x\in E_{\Phi}(\alpha)$ and  $\mu\in V(x)$. Then there is $n_i\rightarrow \infty$ such that
$\mu_{x,n_i}=\frac{1}{n_i}\sum
\limits_{i=0}^{n_i-1}\delta_{T^ix}\rightarrow \mu$. By Lemma
\ref{lem-3.2}, $\mu\in \M(X,T)$ and $\alpha=\lim_{i\rightarrow \infty}
\frac{\log\phi_{n_i}(x)}{n_i}\le \Phi_*(\mu)$. Moreover
$\alpha=\Phi_*(\mu)$ by Lemma \ref{lem-mf-sa}(2).
 Hence $h_\mu(T)\le
t$. It follows that
$$
E_{\Phi}(\alpha)\subset R(t):=\{x\in X:\; \exists\; \mu\in
V(x) \mbox{ with }h_\mu(T)\leq t\}.
$$
By Lemma \ref{Bow}, we have $\htop(T, E_{\Phi}(\alpha))\leq
\htop(T, R(t))\leq t$.  This proves (\ref{e-thm1-3}). Now Theorem \ref{thm-1.1} just follows from (\ref{e-thm1})-(\ref{e-thm1-3}).
\end{proof}

\subsection{A high dimensional version of Theorem \ref{thm-1.2}}

In this subsection, we present and prove a high dimension version of Theorem
\ref{thm-1.2}.

We first give some notation. Let $k\in \N$. For each $i=1,\ldots,k$,
let $\Phi_i=\{\log \phi_{n,i}\}_{n=1}^\infty$ be an asymptotically
sub-additive potential on $(X,T)$. For $\a=(a_1,\ldots,a_k)\in
\mathbb{R}^k$, let
\begin{equation}
\label{e-1.4-h} E_{{\bf \Phi}}(\a)=\{x\in X:\;
\lambda_{\Phi_i}(x)=a_i,\ i=1,2,\ldots,k \}.
\end{equation}
For any $\b=(b_1,\ldots,b_k)\in
\mathbb{R}^k$, define
$$|\b|:=\max\{|b_i|:\;i=1,\ldots,k\}.$$

For ${\bf x}=(x_1,\ldots,x_k)$ and ${\bf y}=(y_1,\ldots,y_k)\in
\mathbb{R}^k$, we write ${\bf x}\geq {\bf y}$ if $x_i\ge y_i$ for all  $1\leq i\leq k$. For $A\subseteq \mathbb{R}^k$,
write
\begin{equation}\label{e-0330a}
\text{cl}_{+}(A)=\{ {\bf x}\in \mathbb{R}^k: \; \exists \;({\bf y}_j)\subset A
\text{ such that } {\bf x}\geq {\bf y}_j \text{ and
}\lim_{j\rightarrow \infty} {\bf y}_j={\bf x}\}.
\end{equation}
For a real valued function $f$ defined on a convex open set $U\subset \R^k$, let $\partial f({\bf x})$, $\partial^e f({\bf x})$,
(${\bf x}\in U$), $\partial f(U)$ and $\partial^e f(U)$ be defined as in Section \ref{S-2.3}. The following result is a high dimensional version of
Theorem \ref{thm-1.2}.

\begin{thm}
\label{thm-1.2-h} Assume  $\htop(T)<\infty$ and
$\overline{\beta}({\bf \Phi})
>-\infty$. Then $P_{{\bf \Phi}}$ is a real continuous convex function on $\mathbb{R}_+^k$. Moreover,
\begin{itemize}
\item[(i)]For any ${\bf t}\in \mathbb{R}^k_+$, if $\a\in
\partial^eP_{{\bf \Phi}}({\bf t})$, then
$$\lim_{\epsilon\to 0}\htop\left(T,\bigcup_{|\b-\a|<\epsilon}
E_{{\bf \Phi}}(\b)\right) =\inf_{{\bf
q}\in \mathbb{R}^k_+}\{P_{{\bf \Phi}}({\bf
q})-\a\cdot{\bf q}\}= P_{{\bf \Phi}}({\bf
t})-\a\cdot {\bf t}.$$ Moreover the first equality is also
valid when $\a\in
\text{\rm cl}_+(\partial^eP_{{\bf \Phi}}(\mathbb{R}^k_+))$.

\item[(ii)] For any ${\bf t}\in \mathbb{R}^k_+$, if $\a\in
\partial P_{{\bf \Phi}}({\bf t})$, then
 $$
\lim_{\epsilon\to 0}\; \sup\left\{h_\mu(T):\ \mu\in
\mathcal{M}(X,T),\;
\left|{\bf \Phi}_*(\mu)-\a\right|<\epsilon\right\}=
\inf_{{\bf q}\in \mathbb{R}^k_+}\{P_{{\bf \Phi}}({\bf
q})-\a\cdot {\bf q}\}.$$ Furthermore the  above equality is
valid  for $\a\in \text{\rm cl}_+(\partial
P_{{\bf \Phi}}(\mathbb{R}^k_+))$.

\item[(iii)] For any
$\a\in \partial P_{{\bf \Phi}}(\mathbb{R}^k_+)\cap {\rm ri}(A)$,
$$
\inf_{{\bf q}\in \mathbb{R}^k_+}\{P_{{\bf \Phi}}({\bf
q})-\a\cdot {\bf q}\}=\sup\{h_\mu(T):\
{\bf \Phi}_*(\mu)=\a\},
 $$
 where $A:=\{\a\in \R^k:\; \a={{\bf \Phi}}_*(\mu)\mbox { for some } \mu\in \M(X,T)\}$, and ${\rm ri}$ denotes the relative
 interior (cf. \cite{Roc70}).
\end{itemize}
\end{thm}

\begin{re}
When $\Phi_i$ ($i=1,\ldots k)$ are all asymptotically additive, the results in Theorem  \ref{thm-1.2-h} can be extended accordingly.
Indeed, one can replace all the terms $\R_+^k$ in Theorem \ref{thm-1.2-h} by $\R^k$, except the two terms in $\text{\rm cl}_+(\partial^eP_{{\bf \Phi}}(\mathbb{R}^k_+))$
 and $\text{\rm cl}_+(\partial P_{{\bf \Phi}}(\mathbb{R}^k_+))$.
\end{re}

To prove Theorem \ref{thm-1.2-h}, we need some preparations. For any
$\a=(a_1,\cdots,a_k)\in \mathbb{R}^k$ and $\epsilon>0$, define
\begin{equation}
\label{e-mfsa-1} G(\a,n,\epsilon):=\left\{ x\in X:\;
\left|\frac{1}{\ell} \log \phi_{\ell,i}(x)-a_i\right|<\epsilon \mbox{ for all }1\leq i\leq k \mbox{ and } \ell\geq n\right\}.
 \end{equation}
We have the following

\begin{lem}
\label{lem-3.3} Assume that $G(\a,n,\epsilon)\neq \emptyset$. Then
\begin{itemize}
\item[(i)]For any ${\bf q}=(q_1,\cdots,q_k)\in \mathbb{R}^k_+$,
$$\htop(T, G(\a,n,\epsilon))\leq P_{{\bf \Phi}}({\bf q})-\sum_{i=1}^k (a_i-\epsilon)q_i.$$
\item[(ii)] Assume  furthermore that all $\Phi_i$ ($i=1,\ldots,k$) are asymptotically additive. Then for any ${\bf q}=(q_1,\cdots,q_k)\in \mathbb{R}^k$,
$$\htop(T, G(\a,n,\epsilon))\leq P_{{\bf \Phi}}({\bf q})-\a\cdot \bq+\epsilon\sum_{i=1}^k|q_i|.$$
\end{itemize}

\end{lem}
\begin{proof} We first prove (i). Fix ${\bf q}=(q_1,\ldots,q_k)\in \mathbb{R}^k_+$. It
suffices to show that for any $s<\htop(T, G(\a,n,\epsilon))$,
$$P_{{\bf \Phi}}({\bf q})\geq s+\sum_{i=1}^k (a_i-\epsilon)q_i.$$ Let $s<\htop(T,
G(\a,n,\epsilon))$ be given. By definition (cf. Section
\ref{S-2.1}), there exists $\gamma>0$ such that $\htop(T,
G(\a,n,\epsilon),\gamma)>s$. Therefore (cf. Section \ref{S-2.1})
$$\infty=M(G(\a,n,\epsilon),s,\gamma)=\lim_{N\to \infty}
M(G(\a,n,\epsilon),s,N,\gamma).$$ Hence there exists $N_0$ such that
$$
M(G(\a,n,\epsilon),s,N,\gamma)\geq 1, \quad \forall\; N\geq N_0.
$$
Now take $N\geq \max\{n, N_0\}$ and let $F$ be a
$(N,\gamma)$-separated subset of $G(\a,n,\epsilon)$ with the maximal
cardinality. Then $\bigcup_{x\in F}B_N(x,\gamma)\supseteq
G(\a,n,\epsilon)$. It follows
\begin{equation}
\label{e-3.6} \# F\cdot \exp(-sN)\geq
M(G(\a,n,\epsilon),s,N,\gamma)\geq 1.
\end{equation}
Since $\sum_{i=1}^k q_i\log \phi_{N,i}(x)\geq
N(\sum_{i=1}^k(a_i-\epsilon)q_i)$ for each $x\in G(\a,n,\epsilon)$,
we have
$$
P_N\left(T,{\bf q}\cdot {\bf \Phi},\gamma\right)\geq \sum_{x\in
F}\exp\Big(\sum_{i=1}^k q_i\log \phi_{N,i}(x)\Big)\geq \# F\cdot
\exp\Big(N\Big(\sum_{i=1}^k(a_i-\epsilon)q_i\Big)\Big).$$ Combining
this with (\ref{e-3.6}) yields $P_N\left(T,{\bf q}\cdot {\bf
\Phi},\gamma\right)\geq \exp(N(s+\sum_{i=1}^k(a_i-\epsilon)q_i))$.
Taking $N\to \infty$ we obtain $P\left(T,{\bf q}\cdot {\bf
\Phi},\gamma\right)\geq s+\sum_{i=1}^k(a_i-\epsilon)q_i$. Hence we
have
$$
P_{{\bf \Phi}}({\bf q})=P\left(T,{\bf q}\cdot {\bf \Phi}\right)\geq
s+\sum_{i=1}^k(a_i-\epsilon)q_i,
$$
which finishes the proof (i).

The proof of (ii) is almost identical. The only difference part is
to use the inequality  $$\sum_{i=1}^k q_i\log \phi_{N,i}(x)\geq
N(\sum_{i=1}^k(a_iq_i-\epsilon |q_i|)$$ for each $x\in
G(\a,n,\epsilon)$ and $\bq\in \R^k$.
\end{proof}

As a corollary, we have
\begin{cor}
\label{cor-3.8} Let $\a=(a_1,\ldots,a_k)\in \mathbb{R}^k$ and
$\epsilon>0$.  Let $E_{{\bf \Phi}}(\a)$ be defined as in {\rm
(\ref{e-1.4-h})}. Then
$$\htop\left(T, \bigcup_{|\b-\a|<\epsilon}
E_{{\bf \Phi}}(\b)\right)\leq P_{{\bf \Phi}}({\bf
q})-\sum_{i=1}^k(a_i-\epsilon)q_i \ \ \text{ for any }{\bf
q}=(q_1,\ldots,q_k)\in \mathbb{R}^k_+,$$ whenever
$\bigcup_{|\b-\a|<\epsilon}E_{{\bf \Phi}}(\b)\neq \emptyset$.
Furthermore if all $\Phi_i$ ($i=1,\ldots,k$) are asymptotically
additive, then
$$\htop\left(T, \bigcup_{|\b-\a|<\epsilon}
E_{{\bf \Phi}}(\b)\right)\leq P_{{\bf \Phi}}({\bf q})-\a\cdot
\bq+\epsilon\sum_{i=1}^k|q_i|, \;\; \forall \; {\bf
q}=(q_1,\ldots,q_k)\in \mathbb{R}^k,$$ whenever
$\bigcup_{|\b-\a|<\epsilon}E_{{\bf \Phi}}(\b)\neq \emptyset$.
\end{cor}

\begin{proof} Observe that
$\bigcup_{|\b-\a|<\epsilon}E_{{\bf \Phi}}(\b)\subseteq
\bigcup_{n=1}^\infty G(\a,n,\epsilon)$. By Proposition
\ref{pro-2.1}, we obtain
\begin{eqnarray*} \htop\left(T,
\bigcup_{|\b-\a|<\epsilon}E_{{\bf \Phi}}(\b)\right) &\leq&
\htop\left(T,\bigcup_{n=1}^\infty G(\a,n,\epsilon)\right)\\
&=& \sup_{n\geq 1}\htop(T,G(\a,n,\epsilon)).
\end{eqnarray*}
By Lemma \ref{lem-3.3}, we obtain the desired result.
 \end{proof}

\begin{lem}
\label{lem-3.9} Assume  $\htop(T)<\infty$ and
$\overline{\beta}({\bf \Phi})>-\infty$. Let ${\bf q}\in
\mathbb{R}^k_+$. Then for any $\epsilon>0$, there exists $\nu\in
\E(X,T)$ such that $h_\nu(T)+{\bf q}\cdot {\bf \Phi}_*(\nu)\geq  P_{{\bf \Phi}}({\bf q})-\epsilon$.
\end{lem}

\begin{proof} Let $\epsilon>0$. By Theorem \ref{thm-3.3}, there
exists $\mu\in \M(X,T)$ such that $$h_\mu(T)+{\bf q}\cdot {\bf \Phi}_*(\mu)\geq P_{{\bf \Phi}}({\bf q})-\epsilon.$$ Let
$\mu=\int_{\E(X,T)}\theta~dm(\theta)$ be the ergodic decomposition of $\mu$.
Then by Proposition \ref{pro-1195}(3), we have
$$\int_{\E(X,T)}(h_\theta(T)+{\bf q}\cdot {\bf \Phi}_*(\theta))~ dm(\theta)=h_\mu(T)+{\bf q}\cdot {\bf \Phi}_*(\mu)\geq
P_{{\bf \Phi}}({\bf q})-\epsilon.$$
Hence there exists at least
one $\nu\in \E(X,T)$ such that $h_\nu(T)+{\bf q}\cdot {\bf \Phi}_*(\nu)\geq P_{{\bf \Phi}}({\bf q})-\epsilon$.
 \end{proof}

\bigskip

 The following result is important  in the proof of Theorem \ref{thm-1.2-h}.

\begin{lem}
\label{lem-mfsa-5} Assume   $\htop(T)<\infty$ and
$\overline{\beta}({\bf \Phi})>-\infty$. Let ${\bf t}\in
\mathbb{R}^k_+$. Assume $\a\in
\partial^e P_{{\bf \Phi}}({\bf t})$. Then for any
$\epsilon>0$, there exists $\nu\in \mathcal{E}(X,T)$ such that
$|{\bf \Phi}_*(\nu)-\a|<\epsilon$ and
$|h_\nu(T)-(P_{{\bf \Phi}}({\bf t})-\a\cdot {\bf t})|<\epsilon$.
\end{lem}
\begin{proof} By Proposition \ref{pro-mfsa-1}, $P_{{\bf \Phi}}$ is a real  continuous
convex function on $\mathbb{R}^k_+$. Let ${\bf t}=(t_1,\ldots,t_k)\in \mathbb{R}^k_+$
and $\epsilon>0$. We first assume that
$P_{{\bf \Phi}}$ is differentiable at ${\bf t}$. Let $\a=P_{{\bf \Phi}}'({\bf t})$ and write $\a=(a_1,\ldots,a_k)$. Set
$\delta=\min\{\frac{\epsilon}{3k},\frac{\epsilon}{3k|{\bf t}|}\}$.
Choose $\gamma_0>0$ such that
\begin{equation}\label{e-mfsa-3}
\frac{\left|P_{{\bf \Phi}}({\bf t}+{\bf s})
-P_{{\bf \Phi}}({\bf t})-\a\cdot {\bf s} \right|}{|{\bf s}|}<\delta\quad \mbox{ for all }{\bf s}\in \R^k\mbox{ with } 0<|{\bf
s}|\leq \gamma_0.
\end{equation}
Pick $\eta$ such that $0<\eta\leq\min\{\epsilon/3,
\delta\gamma_0\}$. By Lemma \ref{lem-3.9}, there exists $\nu\in
\mathcal{E}(X,T)$ such that
\begin{equation}\label{e-mfsa-4}
h_\nu(T)+ {\bf t}\cdot {\bf \Phi}_*(\nu)\geq
P_{{\bf \Phi}}({\bf t})-\eta.
\end{equation}
Meanwhile by Theorem \ref{thm-3.3},
\begin{equation}\label{e-mfsa-5}
h_\nu(T)+({\bf t}+{\bf s})\cdot {\bf \Phi}_*(\nu)\leq
P_{{\bf \Phi}}({\bf t}+{\bf s}) \quad \mbox{ for all }{\bf s}\in \R^k\mbox{ with }{\bf t}+{\bf s}\in \R^k_+.
\end{equation}
Combining (\ref{e-mfsa-5}) and (\ref{e-mfsa-4}) yields
\begin{equation}\label{e-new1}
P_{{\bf \Phi}}({\bf t}+{\bf
s})-P_{{\bf \Phi}}({\bf t})\geq {\bf s}\cdot {\bf \Phi}_*(\nu)-\eta \quad \mbox{ for all }{\bf s}\in \R^k
\mbox{ with }{\bf t}+{\bf s}\in \R^k_+.
\end{equation}
Construct points ${\bf s}_i\in \R^k$ ($i=1,\ldots,k$) by ${\bf
s}_i=(s_{i,1},\ldots, s_{i,k})$, where
$$s_{i,j}=\begin{cases} 0 &\text{ if } i\neq j,\\ \gamma_0 &\text{ if }
i=j.\end{cases}$$
 Taking ${\bf s}=\pm {\bf s}_i$ in (\ref{e-new1}) yields
$$
\frac{P_{{\bf \Phi}}({\bf t}+{\bf
s}_i)-P_{{\bf \Phi}}({\bf t})}{\gamma_0}\geq
(\Phi_i)_*(\nu)-\frac{\eta}{\gamma_0}\quad \mbox{ and }\quad
\frac{P_{{\bf \Phi}}({\bf t}-{\bf
s}_i)-P_{{\bf \Phi}}({\bf t})}{-\gamma_0}\leq
(\Phi_i)_*(\nu)+\frac{\eta}{\gamma_0}.$$
Combining  the above  two inequalities  with
(\ref{e-mfsa-3}), we have
$$
|(\Phi_i)_*(\nu)-a_i|\leq
\delta+\frac{\eta}{\gamma_0}\leq 2\delta<\epsilon, \quad i=1,\ldots,k,
$$
which combining  with (\ref{e-mfsa-4}) and (\ref{e-mfsa-5})  yields
$$h_\nu(T)\geq P_{{\bf \Phi}}({\bf t})-\eta-\sum_{i=1}^k t_i(a_i+2\delta)
\geq P_{{\bf \Phi}}({\bf t})-\a\cdot {\bf t}-\eta-2k|{\bf t}|\delta
>P_{{\bf \Phi}}({\bf t})-\a\cdot {\bf t}-\epsilon.
$$
and
$$h_\nu(T)\leq P_{{\bf \Phi}}({\bf t})-{\bf t}\cdot {\bf \Phi}_*(\nu)
\leq  P_{{\bf \Phi}}({\bf t})-\sum_{i=1}^k t_i(a_i
-2\delta)
 < P_{{\bf \Phi}}({\bf t})-\a\cdot {\bf t}+\epsilon.
$$
This proves the lemma in the case that
$P_{{\bf \Phi}}$ is differentiable at ${\bf t}$.

Now assume that $P_{{\bf \Phi}}$ is not differentiable
at ${\bf t}$. Let $\a=(a_1,\ldots,a_k)\in
\partial^e P_{{\bf \Phi}}({\bf t})$. Since $P_{{\bf \Phi}}$ is a real continuous convex function on $\mathbb{R}_+^k$,  by Proposition \ref{lem-key2.2}, there
exists a sequence $({\bf t}_n)\subset \R^k_+$ converging to ${\bf t}$ such that
$P_{{\bf \Phi}}^\prime({\bf t}_n)$ exists for each $n$
and $\lim \limits_{n\to \infty}P_{\bf \Phi}^\prime({\bf t}_n)=\a$. Choose a large integer $n$ such that
\begin{equation}
\label{e-mfsa-6} \left|P_{{\bf \Phi}}^\prime({\bf t}_n)-\a\right|<\frac{\epsilon}{2} \quad \mbox{and} \quad
\left|\left(P_{{\bf \Phi}}({\bf t}_n)-P_{{\bf \Phi}}^\prime({\bf t}_n)\cdot {\bf t}_n\right)-\left(P_{{\bf \Phi}}({\bf
t})-\a\cdot {\bf
t}\right)\right|<\frac{\epsilon}{2}.
\end{equation} As  proved in the last paragraph, we can
choose $\nu\in \mathcal{E}(X,T)$ such that
\begin{equation}
\label{e-mfsa-7}
\left|{\bf \Phi}_*(\nu)-P_{\bf \Phi}^\prime({\bf t}_n)\right|<\frac{\epsilon}{2}\quad \mbox{and}\quad
\left|h_\nu(T)-\left(P_{{\bf \Phi}}({\bf t}_n)-P^\prime_{{\bf \Phi}}({\bf t}_n)\cdot {\bf t}_n\right)\right|
<\frac{\epsilon}{2}.
\end{equation}
Combining (\ref{e-mfsa-6}) with (\ref{e-mfsa-7}) yields
$|{\bf \Phi}_*(\nu)-\a|<\epsilon$ and
$|h_\nu(T)-(P_{{\bf \Phi}}({\bf
t})-\a\cdot {\bf t})|<\epsilon$. This finishes
the proof of the lemma.
\end{proof}

\begin{proof}[Proof of Theorem \ref{thm-1.2-h}]
By Proposition \ref{pro-mfsa-1}, $P_{{\bf \Phi}}$ is a real continuous convex function  on $\mathbb{R}_+^k$.

We first prove part (i) of the theorem. Let ${\bf t}=(t_1,\ldots, t_k)\in \mathbb{R}_+^k$ and  $\a=(a_1,\ldots, a_k)\in
\partial^e P_{{\bf \Phi}}({\bf t})$. Let
$\epsilon>0$. Then by Lemma \ref{lem-mfsa-5}, there exists $\nu\in
\E(X,T)$ such that
\begin{equation}\label{e-mfsa-8}
\left|{\bf \Phi}_*(\nu)-\a\right|<\epsilon\quad
\mbox{and} \quad \left|h_\nu(T)-(P_{{\bf \Phi}}({\bf
t})-\a\cdot {\bf t})\right|<\epsilon.
\end{equation}
Since $\nu\in \E(X,T)$, by Proposition \ref{pro-1195}(1),
$\lambda_{\Phi_i}(x)=(\Phi_i)_*(\nu)$ for $\nu$-a.e.\! $x\in X$, $i=1,\ldots,k$.
That is,
$\nu(E_{{\bf \Phi}}({\bf \Phi}_*(\nu)))=1$.
By Proposition \ref{pro-2.1}(3),
$h_{\text{top}}\left(T,E_{{\bf \Phi}}({\bf \Phi}_*(\nu))\right)\geq
 h_\nu(T)$. Since ${\bf \Phi}_*(\nu)$ and $h_\nu(T)$
satisfy (\ref{e-mfsa-8}), we have
$$h_{\text{top}}\left(T,\bigcup_{|\b-\a|<\epsilon}
E_{{\bf \Phi}}(\b)\right)\geq
h_{\text{top}}\left(T,
E_{{\bf \Phi}}({\bf \Phi}_*(\nu))\right)\geq
h_\nu(T)\geq P_{{\bf \Phi}}({\bf t})-\a\cdot {\bf t}-\epsilon.
$$
On the other hand by Corollary \ref{cor-3.8}, we have
$$h_{\text{top}}\left(T,\bigcup_{|\b-\a|<\epsilon}
E_{{\bf \Phi}}(\b)\right) \leq
P_{{\bf \Phi}}({\bf t})-\sum_{i=1}^k(a_i-\epsilon)
t_i.
$$
Combining the above two inequalities and letting $\epsilon\to 0$,
we obtain
$$\lim_{\epsilon\to 0}
h_{\text{top}}\left(T,\bigcup_{|\b-\a|<\epsilon}
E_{{\bf \Phi}}(\b)\right) =
P_{{\bf \Phi}}({\bf t})-\a\cdot {\bf t}=\inf_{{\bf q}\in
\mathbb{R}^k_+}\{P_{{\bf \Phi}}({\bf
q})-\a\cdot {\bf q}\}.$$

Now assume $\a\in \text{cl}_+(\partial^e
P_{{\bf \Phi}}(\R^k))$. Then there exist ${\bf t}_j\in
\mathbb{R}_+^k$ and $\b_j\in
\partial^e P_{{\bf \Phi}}({\bf t}_j)$, $j\in
\mathbb{N}$ such that $\a\geq\b_j$
and $\lim \limits_{j\rightarrow
\infty}\b_j=\a$. Let $\epsilon>0$. There exists a large $j_\epsilon$ such that
$|\a-\b_{j_\epsilon}|<\frac{\epsilon}{2}$.
Thus
\begin{eqnarray*}\htop\left(T,\bigcup_{|\b-\a|<\epsilon}
E_{{\bf \Phi}}(\b)\right)&\geq&
\htop\left(T,\bigcup_{|\b-\b_{j_\epsilon}|<\frac{\epsilon}{2}}
E_{{\bf \Phi}}(\b)\right)\\
&\ge&P_{{\bf \Phi}}({\bf t}_{j_\epsilon})
-\b_{j_\epsilon}\cdot {\bf t}_{j_\epsilon}\geq
P_{{\bf \Phi}}({\bf t}_{j_\epsilon})
-\a\cdot {\bf t}_{j_\epsilon}\\
&\geq& \inf_{{\bf q}\in
\mathbb{R}^k_+}\{P_{{\bf \Phi}}({\bf
q})-\a\cdot {\bf q}\},
\end{eqnarray*}
 and hence
$$\lim_{\epsilon\to 0} \htop\left(T,\bigcup_{|\b-\a|<\epsilon}
E_{{\bf \Phi}}(\b)\right)\geq \inf_{{\bf
q}\in \mathbb{R}^k_+}\{P_{{\bf \Phi}}({\bf
q})-\a\cdot {\bf q}\}.$$ Meanwhile, the upper bound follows
from Corollary \ref{cor-3.8}. This finishes the proof of part (i).

\medskip
To show (ii), we first prove the following upper bound:
\begin{equation}
\label{e-mfsa-9} \lim_{\epsilon\to 0}\; \sup\left\{h_\mu(T):\ \mu\in
\mathcal{M}(X,T),\;
\left|{\bf \Phi}_*(\mu)-\a\right|<\epsilon\right\}\leq
\inf_{{\bf q}\in \mathbb{R}^k_+}\{P_{{\bf \Phi}}({\bf
q})-\a\cdot {\bf q}\}
\end{equation}
for any $\a=(a_1,\ldots,a_k)\in \mathbb{R}^k$, where we take the
convention $\sup\emptyset =-\infty$. To see it, let ${\bf q}=(q_1,\ldots,q_k)\in
\mathbb{R}^k_+$ and $\epsilon>0$. Then by Theorem
\ref{thm-3.3}, for any $\mu\in \mathcal{M}(X,T)$ satisfying
$\left|{\bf \Phi}_*(\mu)-\a\right|<\epsilon$,
$$h_\mu(T)\leq P_{{\bf \Phi}}({\bf q})-{\bf q}\cdot {\bf \Phi}_*(\mu) \leq P_{{\bf \Phi}}({\bf q})-
\sum_{i=1}^k(a_i-\epsilon)q_i.$$ That is,
$\sup\left\{h_\mu(T):\ \mu\in \mathcal{M}(X,T),\;
\left|{\bf \Phi}_*(\mu)-\a\right|<\epsilon\right\}\leq
P_{{\bf \Phi}}({\bf q})-\sum_{i=1}^k(a_i-\epsilon)
q_i$. Letting $\epsilon\to 0$ yields (\ref{e-mfsa-9}).

Now we prove the lower bound. Assume $\a\in
\partial P_{{\bf \Phi}}({\bf t})$
for some ${\bf t}\in \mathbb{R}^k_+$. Let $\epsilon>0$. Then by Minkowski's Theorem (cf. Section \ref{S-2.3}), there exist  $\a_j\in \partial^e
P_{{\bf \Phi}}({\bf t})$ and  $\lambda_j\in [0,1]$, $j=1,\ldots,k+1$,
 such that $\sum_{j=1}^{k+1} \lambda_j=1$ and
\begin{equation}\label{eq-aaa1}
\a=\sum_{j=1}^{k+1} \lambda_j\a_j.
\end{equation}
By Lemma \ref{lem-mfsa-5}, there exist $\nu_j\in \E(X,T)$,
$j=1,\ldots,k+1$ such that
\begin{equation}
\label{e-mfsa-10}
\left|{\bf \Phi}_*(\nu_j)-\a_j\right|<\epsilon,
\quad\left|h_{\nu_i}(T)-(P_{{\bf \Phi}}({\bf
t})-\a_j\cdot {\bf t})\right|<\epsilon.
\end{equation}
Set $\nu=\sum_{j=1}^{k+1} \lambda_j \nu_j$. Then $\nu\in
\mathcal{M}(X,T)$ and
$$
{\bf \Phi}_*(\nu)=\sum_{j=1}^{k+1}\lambda_j{\bf \Phi}_*(\nu_j),\quad
h_\nu(T)=\sum_{j=1}^{k+1} \lambda_j h_{\nu_j}(T).
$$
Combining these  with \eqref{eq-aaa1} and (\ref{e-mfsa-10})  yields
$$
\left|{\bf \Phi}_*(\nu)-\a\right|<
\epsilon,\quad\left|h_{\nu}(T)-(P_{{\bf \Phi}}({\bf
t})-\a\cdot {\bf t})\right|<\epsilon.
$$
Thus $\sup\left\{h_\mu(T):\; \mu\in \mathcal{M}(X,T),\;
\left|{\bf \Phi}_*(\mu)-\a\right|<\epsilon\right\}
\geq h_\nu(T)\geq P_{{\bf \Phi}}({\bf
t})-\a\cdot {\bf t}-\epsilon$. Letting
$\epsilon\to 0$ yields the desired lower bound
\begin{align*}
& \lim_{\epsilon\to 0}\; \sup\left\{h_\mu(T)\; :\mu\in
\mathcal{M}(X,T),\;
\left|{\bf \Phi}_*(\mu)-\a\right|<\epsilon\right\}\\
&\quad \geq P_{{\bf \Phi}}({\bf t})-\a\cdot {\bf t}=\inf_{{\bf q}\in \mathbb{R}^k_+}\{P_{{\bf \Phi}}({\bf
q})-\a\cdot {\bf q}\}.
\end{align*}

In the end, we  assume $\a\in \text{\rm cl}_+(\partial
P_{{\bf \Phi}})$. Then there exist ${\bf t}_j\in
\mathbb{R}_+^k$ and $\b_j\in
\partial^e P_{{\bf \Phi}}({\bf t}_j)$, $j\in
\mathbb{N}$ such that $\a\geq\b_j$
 and $\lim \limits_{j\rightarrow
\infty}\b_j=\a$. Let $\epsilon>0$. There exists a large $j_\epsilon$ such that
$|\a-\b_{j_\epsilon}|<\frac{\epsilon}{2}$.
Thus
\begin{eqnarray*}
& &\sup\left\{h_\mu(T)\; :\mu\in \mathcal{M}(X,T),\;
\left|{\bf \Phi}_*(\mu)-\a\right|<\epsilon\right\}\\
&& \quad \geq\sup\left\{h_\mu(T)\; :\mu\in \mathcal{M}(X,T),\;
\left|{\bf \Phi}_*(\mu)-\b_{j_\epsilon}\right|
<\frac{\epsilon}{2}\right\}\\
&& \quad \geq  P_{{\bf \Phi}}({\bf t}_{j_\epsilon})
-\b_{j_\epsilon}\cdot {\bf t}_{j_\epsilon}\geq
P_{{\bf \Phi}}({\bf t}_{j_\epsilon})
-\a\cdot {\bf t}_{j_\epsilon}\\
&&\quad \geq  \inf_{{\bf q}\in
\mathbb{R}^k_+}\{P_{{\bf \Phi}}({\bf
q})-\a\cdot {\bf q}\},
\end{eqnarray*}
 and hence
$$\lim_{\epsilon\to 0}\; \sup\left\{h_\mu(T)\; :\mu\in
\mathcal{M}(X,T),\;
\left|{\bf \Phi}_*(\mu)-\a\right|<\epsilon\right\}\\
\geq \inf_{{\bf q}\in \mathbb{R}^k_+}\{P_{{\bf \Phi}}({\bf
q})-\a\cdot {\bf q}\}.$$
 This finishes the proof of part
(ii).

To show part (iii), let $A=\{\a\in \R^k:\; \a={\bf \Phi}_*(\mu) \mbox{ for some }\mu\in \M(X,T)\}$. Clearly, $A$ is non-empty and convex. Define $g:\; A\to \R$ by
$$
g(\a)=\sup\{h_\mu(T):\; \mu\in \M(X,T),\; {\bf \Phi}_*(\mu)=\a\}.
$$
Then $g$ is a real-valued concave function on $A$. Take
$$
W({\bf x} )=\sup\{g(\a)+\a\cdot {\bf x}:\; \a\in A\},\quad \forall \; {\bf x}\in \R^k.
$$
 Apply Corollary \ref{cor-0330} to obtain
\begin{equation}
\label{e-0330c}
\inf\{W({\bf x})-\a\cdot {\bf x}:\; {\bf x}\in \R^k\}=g(\a),\quad \forall \; \a\in {\rm ri}(A).
\end{equation}
However, by Theorem \ref{thm-3.3}, we have $P_{\bf \Phi}({\bf q})=W({\bf q})$ for all ${\bf q}\in \R^k_+$. Now assume that
$\a\in \partial P_{\bf \Phi}({\bf q})\cap {\rm ri}(A)$ for some ${\bf q}\in \R^k_+$. Then $\a\in \partial W({\bf q})\cap {\rm ri}(A)$. Hence
\begin{eqnarray*}
g(\a)&=&\inf\{W({\bf x})-\a\cdot {\bf x}:\; {\bf x}\in \R^k\}=W({\bf q})-\a\cdot {\bf q}\\
&=&P_
{\bf \Phi}({\bf q})-\a\cdot {\bf q}=\inf\{P_{\bf \Phi}({\bf x})-\a\cdot {\bf x}:\; {\bf x}\in \R_+^k\}.
\end{eqnarray*}
This finishes the proof of (iii) and hence the proof of Theorem \ref{thm-1.2-h}.
\end{proof}

\begin{proof}[Proof of Theorem \ref{thm-1.2}] Here we show that Theorem \ref{thm-1.2} is just the one-dimensional version of Theorem \ref{thm-1.2-h}.
To see it, let $(X,T)$ be a TDS with $\htop(T)<\infty$ and let $\Phi=\{\log \phi_n\}_{n=1}^\infty$
be an asymptotically sub-additive potential  on $X$  satisfying $\overline{\beta}(\Phi)
>-\infty$. Let  $t>0$.  It is clear that $\partial^e P_{\Phi}(t)=\{
P'_{\Phi}(t-), P'_{\Phi}(t+)\}$ and $\partial
P_{\Phi}(t)=[P'_{\Phi}(t-), P'_{\Phi}(t+)]$.
Thus
\begin{equation}
\label{e-0331a}
\partial P_{\Phi}(\R_+)=\bigcup_{t>0} [P'_{\Phi}(t-),
P'_{\Phi}(t+)]\quad\mbox{  and }\quad \partial^e P_{\Phi}(\R_+)=\bigcup_{t>0}
\{P'_{\Phi}(t-), P'_{\Phi}(t+)\}.
\end{equation}
 Moreover,
\begin{equation}
\label{e-0331b}
\text{\rm cl}_+(\partial
P_{\Phi}(\R_+))=\partial P_{\Phi}(\R_+)\cup \{
P'_{\Phi}(\infty)\} \text{ and }\text{\rm cl}_+(\partial^e
P_{\Phi}(\R_+))=\partial^e P_{\Phi}(\R_+)\cup \{
P'_{\Phi}(\infty)\}.
\end{equation}
Furthermore,  by Lemma \ref{lem-mf-sa}(4),
$q\overline{\beta}(\Phi)\leq P_{\Phi}(q)\leq \htop(T)+
q\overline{\beta}(\Phi)$, from which we obtain
$P'_{\Phi}(\infty):=\lim_{q\to\infty}P_{\Phi}(q)/q=\overline{\beta}(\Phi)$. By the way, applying Theorem
\ref{thm-1.2-h}(ii), for each  $a\in \partial P_{\Phi}(t)$ and any $\epsilon>0$, there exists $\mu\in \M(X,T)$ such that
$|\Phi_*(\mu)-a|<\epsilon$. It implies that
\begin{equation}
\label{e-0331c}
(\lim_{t\to 0+}P_\Phi'(t-), P_\Phi^\prime(\infty))\subseteq {\rm int}(A)={\rm ri}(A),
\end{equation}
where $A:=\{a\in \R: \; a=\Phi_*(\mu) \mbox{ for some }
\mu\in \M(X,T)\}$. According to (\ref{e-0331a})-(\ref{e-0331c}), Theorem \ref{thm-1.2} just follows from
 Theorem
\ref{thm-1.2-h}.
\end{proof}

\subsection{A high-dimensional version of  Theorem \ref{thm-1.3}}

Let $\Phi_i=\{\log \phi_{n,i}\}_{n=1}^\infty$ ($i=1,\ldots,k$) be
asymptotically sub-additive potentials on a TDS $(X,T)$. Let ${\bf
\Phi}=(\Phi_1,\ldots,\Phi_k)$. For $\delta>0$, we define
\begin{equation}
\label{e-newno}
\text{cl}^\delta_+(\partial P_{{\bf \Phi}}(\R_+^k)):=
\text{cl}_+(\bigcup_{\{{\bf t}\in \mathbb{R}^k_+: \;t_i\ge \delta,\;
i=1,2,\cdots,k\}}\partial P_{{\bf \Phi}}({\bf t})),
\end{equation}
where $\text{cl}_{+}(A)$ is defines as in (\ref{e-0330a}).

The following theorem is a high dimensional version of Theorem  \ref{thm-1.3}.
\begin{thm}
\label{thm-1.3-h} Assume $\htop(T)<\infty$, $\overline{\beta}({\bf \Phi})>-\infty$, and that  the entropy map $\mu\mapsto h_\mu(T)$ is upper
semi-continuous on $\M(X,T)$. Then
\begin{itemize}
\item[(i)]For any ${\bf t}\in \mathbb{R}_+^k$, if $\a\in \partial^e
P_{{\bf \Phi}}({\bf t})$, then
$E_{{\bf \Phi}}(\a)\neq \emptyset$ and
$$
\htop(T,E_{{\bf \Phi}}(\a)) =\inf_{{\bf q}\in
\mathbb{R}_+^k}\{P_{{\bf \Phi}}({\bf q})-\a\cdot {\bf q}\}=P_{{\bf \Phi}}({\bf t})-\a\cdot {\bf t}.
$$

\item[(ii)] For $\a\in \bigcup_{\delta>0} \text{\rm cl}^\delta_+(\partial
P_{{\bf \Phi}}(\R_+^k))$,
 $$
\inf_{{\bf q}\in \mathbb{R}_+^k}\{P_{{\bf \Phi}}({\bf
q})-\a\cdot {\bf q}\}=\max\{h_\mu(T):\; \mu\in
\M(X,T),\;{\bf \Phi}_*(\mu)=\a\}.
  $$

\item[(iii)] If~ ${\bf t}\in \mathbb{R}_+^k$ such that
${\bf t}\cdot {\bf \Phi}$ has a unique equilibrium
state $\mu_{\bf t}\in \mathcal{M}(X,T)$, then $\mu_{\bf t}$ is
ergodic, $P^\prime_{{\bf \Phi}} ({\bf
t})={\bf \Phi}_*(\mu_{\bf t})$,
$E_{{\bf \Phi}}(P^\prime_{{\bf \Phi}} ({\bf
t}))\neq\emptyset$ and
$\htop(T,E_{{\bf \Phi}}(P^\prime_{{\bf \Phi}} ({\bf
t})))=h_{\mu_{\bf t}}(T)$.
\end{itemize}
\end{thm}
\begin{proof}
We first prove (i). Fix ${\bf t}\in \mathbb{R}^k_+$. By Theorem
\ref{lem-mfsa1-1}, $\mathcal{I}({\bf \Phi},{\bf t})$ is
a non-empty compact convex subset of $\mathcal{M}(X,T)$. We claim
that for any $\b\in \mathbb{R}^k$ the set
$$\mathcal{I}_{\b}({\bf \Phi},{\bf t}):=
\{\nu\in \mathcal{I}({\bf \Phi},{\bf t}):\
{\bf \Phi}_*(\nu)=\b\}$$ is compact and
convex. The convexity is clear. To show the compactness, assume that
$\{\nu_n\}\subset
\mathcal{I}_{\b}({\bf \Phi},{\bf t})$
and $\nu_n$ converges to $\nu$ in $\mathcal{M}(X,T)$. Then by the
upper semi-continuity of $h_{(\cdot)}(T)$ and $(\Phi_i)_*(\cdot)$,  $i=1,\ldots,k$, we have
$$h_\nu(T)+{\bf t}\cdot{\bf \Phi}_*(\nu)\geq \lim_{n\to
\infty}h_{\nu_n}(T)+{\bf t}\cdot {\bf \Phi}_*(\nu_n)=P_{{\bf \Phi}}({\bf t}).$$
By Theorem \ref{thm-3.3}, $\nu\in
\mathcal{I}({\bf \Phi},{\bf t})$ and furthermore,
$h_\nu(T)=h_{\nu_n}(T)=P_{{\bf \Phi}}({\bf
t})-{\bf t}\cdot \b$ and $
{\bf \Phi}_*(\nu)={\bf \Phi}_*(\nu_n)=\b$.
This is,
 $\nu\in \mathcal{I}_{\b}({\bf \Phi},{\bf t})$.
Hence $\mathcal{I}_{\b}({\bf \Phi},{\bf
t})$ is compact. This finishes the proof of the claim.

Now let $\a\in \partial^e P_{{\bf \Phi}}({\bf
t})$. By  Theorem \ref{lem-mfsa1-1}, the set
$\mathcal{I}_{\a}({\bf \Phi},{\bf t})$
is non-empty. We are going to show further that  $\mathcal{I}_{\a}({\bf \Phi},{\bf t})$ contains at least one  ergodic
measure. Since
$\mathcal{I}_{\a}({\bf \Phi},{\bf t})$
 is a non-empty
compact convex subset of $\mathcal{M}(X,T)$, by the Krein-Milman
theorem (c.f. \cite[p. 146]{Con-book}), it contains at least one
extreme point, denoted by $\nu$. Let $\nu=p\nu_1+(1-p)\nu_2$ for
some $0<p<1$ and $\nu_1,\nu_2\in \mathcal{M}(X,T)$. Then
\begin{align*}
P_{{\bf \Phi}}({\bf t})&=h_\nu(T)+{\bf t}\cdot {\bf \Phi}_*(\nu) \\
&=p(h_{\nu_1}(T)+{\bf t}\cdot {\bf \Phi}_*(\nu_1))+(1-p)(h_{\nu_2}(T)+{\bf t}\cdot {\bf \Phi}_*(\nu_2)).
\end{align*}
By Theorem \ref{thm-3.3}, $\nu_1,\nu_2\in
\mathcal{I}({\bf \Phi},{\bf t})$. By Theorem
\ref{lem-mfsa1-1},
${\bf \Phi}_*(\nu_1),{\bf \Phi}_*(\nu_2)\in
\partial P_{{\bf \Phi}}({\bf t})$. Moreover,
note that
$\a={\bf \Phi}_*(\nu)=p{\bf \Phi}_*(\nu_1)
+(1-p){\bf \Phi}_*(\nu_2)$, we have
${\bf \Phi}_*(\nu_1)={\bf \Phi}_*(\nu_2)=\a$ since
$\a\in \partial^e P_{{\bf \Phi}}({\bf t})$. That
is, $\nu_1,\nu_2\in
\mathcal{I}_{\a}({\bf \Phi},{\bf t})$.
Since $\nu$ is an extreme point of
$\mathcal{I}_{\a}({\bf \Phi},{\bf t})$,
we have $\nu_1=\nu_2=\nu$. It follows that $\nu$ is an extreme point
of $\mathcal{M}(X,T)$, i.e., $\nu$ is ergodic. By Proposition
\ref{pro-1195}(1), we have
$\nu(E_{{\bf \Phi}}(\a))=1$, and thus
by Proposition \ref{pro-2.1}(3),
$$h_{\text{top}}(T,E_{{\bf \Phi}}(\a))
\geq h_{\nu}(T)=P_{{\bf \Phi}}({\bf t})-{\bf t}\cdot {\bf a}
=\inf_{{\bf q}\in
\mathbb{R}^k_+}\{P_{{\bf \Phi}}({\bf q})-{\bf q}\cdot \a\}.$$ However by Corollary \ref{cor-3.8}, the upper bound
$h_{\text{top}}(T,E_{{\bf \Phi}}(\a))\leq
\inf_{{\bf q}\in \mathbb{R}^k_+}\{ P_{{\bf \Phi}}({\bf
q})-{\bf q}\cdot \a\}$ is generic. Thus we have the
equality
$$
h_{\text{top}}(T,E_{{\bf \Phi}}(\a))=h_{\nu}(T)
=P_{{\bf \Phi}}({\bf t})-{\bf t}\cdot \a=
\inf_{{\bf q}\in \mathbb{R}^k_+}\{ P_{{\bf \Phi}}({\bf
q})-{\bf q}\cdot \a\}.
$$
This finishes the  proof of  part (i).

\medskip
To show (ii), by Theorem \ref{thm-1.2-h}(ii) we need only to show
that for $\a\in \text{cl}_+^\delta(\partial
P_{{\bf \Phi}})$,
 $$
\inf_{{\bf q}\in \mathbb{R}_+^k}\{P_{{\bf \Phi}}({\bf
q})-{\bf a}\cdot {\bf q}\}\le \max\{h_\mu(T):\; \mu\in
\M(X,T),\;{\bf \Phi}_*(\mu)=\a\}.
$$
To see it, we first assume that $\a\in
\partial P_{{\bf \Phi}}({\bf t})$
for some ${\bf t}\in \mathbb{R}^k_+$.  By Theorem
\ref{thm-1.2-h}(ii), there exists $(\nu_j)\subset \mathcal{M}(X,T)$ such
that
$$\lim_{j\rightarrow \infty}
{\bf \Phi}_*(\nu_j)=\a \text{ and
}\limsup_{j\rightarrow \infty}h_{\nu_j}(T)\ge \inf_{{\bf q}\in
\mathbb{R}_+^k}\{P_{{\bf \Phi}}({\bf q})-{\bf a}\cdot {\bf q}\}.$$ Extract a subsequence if necessary so that
$\lim_{j\rightarrow \infty}\nu_j=\nu$ for some $\nu\in
\mathcal{M}(X,T)$. Then
$$h_\nu(T)\ge \limsup_{j\rightarrow
\infty}h_{\nu_j}(T)\ge \inf_{{\bf q}\in
\mathbb{R}_+^k}\{P_{{\bf \Phi}}({\bf q})-{\bf a}
\cdot {\bf q}\}=P_{{\bf \Phi}}({\bf t})-\a \cdot {\bf t}$$
and ${\bf \Phi}_*(\nu)\ge \limsup_{j\rightarrow
\infty}{\bf \Phi}_*(\nu_j)=\a$
by the
upper-semi continuity of $h_{(\cdot)}(T)$ and $(\Phi_i)_*(\cdot)$. Hence
$$h_\nu(T)\ge P_{{\bf \Phi}}({\bf t})-\a \cdot {\bf t}
\ge P_{{\bf \Phi}}({\bf t})-{\bf \Phi}_*(\nu)\cdot {\bf t}\ge h_\nu(T),$$ which implies
${\bf \Phi}_*(\nu)=\a$ and $h_\nu(T)=
P_{{\bf \Phi}}({\bf t})-\a\cdot {\bf t}$.

Next we assume that $\a\in \text{cl}^\delta_+(\partial
P_{\bf \Phi}(\R_+^k))$ for some $\delta>0$. Then there exists a sequence
$({\bf t}_j)\in \mathbb{R}_+^k$ such that each entry of ${\bf t}_j$ is greater than $\delta$, and there exists
$\a_j\in
\partial^e P_{{\bf \Phi}}({\bf t}_j)$ for each $j$  such that $\a\geq\a_j$
 and $\lim \limits_{j\rightarrow
\infty}\a_j=\a$. By the above
discussion, for each $j\in \mathbb{N}$ there exists $\mu_j\in
\mathcal{M}(X,T)$ such that
${\bf \Phi}_*(\mu_j)=\a_j$  and $h_{\mu_j}(T)=
P_{{\bf \Phi}}({\bf t}_j)-\a_j\cdot {\bf  t}_j$.
Extract a subsequence if necessary so that  $\lim_{j\rightarrow
\infty}\mu_j=\mu$ for some $\mu\in \mathcal{M}(X,T)$. Thus
${\bf \Phi}_*(\mu)\geq \lim_{j\to \infty}{\bf \Phi}_*(\mu_j)=\a$ and
\begin{align*}
h_\mu(T)&\ge \limsup_{j\rightarrow \infty} h_{\mu_j}(T)=
\limsup_{j\rightarrow \infty} (P_{{\bf \Phi}}({\bf
t}_j)-\a_j\cdot {\bf  t}_j)\\
&\ge \limsup_{j\rightarrow \infty} (P_{{\bf \Phi}}({\bf
t}_j)-\a\cdot {\bf  t}_j)\\
&\ge \limsup_{j\rightarrow
\infty} ( h_\mu(T)+({\bf \Phi}_*(\mu)-\a)\cdot {\bf t_j} )\\
&\ge h_\mu(T)+\sum_{i=1}^k((\Phi_i)_*(\mu)-a_i)\delta.
\end{align*}
This implies that ${\bf \Phi}_*(\mu)=\a$ and
$$h_\mu(T)
\ge \limsup_{j\rightarrow \infty} (P_{{\bf \Phi}}({\bf
t}_j)-\a\cdot {\bf  t}_j)
\ge
\inf_{{\bf q}\in \mathbb{R}^k_+}\{ P_{{\bf \Phi}}({\bf
q})-\a\cdot {\bf q}\}.$$
This finishes the proof of part (ii).

\medskip
Now we turn to prove (iii). We assume ${\bf t}\in
\mathbb{R}^k_+$ such that ${\bf t}\cdot {\bf \Phi}$ has
a unique equilibrium state $\mu_{\bf t}$. By Theorem
\ref{lem-mfsa1-1}, $\partial P_{{\bf \Phi}}({\bf t})=\{
{\bf \Phi}_*(\mu_t)\}$. Now (iii) comes from  parts (i)
and (ii) of the theorem.
\end{proof}

\begin{proof}[Proof of Theorem \ref{thm-1.3}]
It follows directly from Theorem \ref{thm-1.3-h},  using the fact that in the one-dimensional case, 
$\partial^eP_{\Phi}(t)=\{P_\Phi'(t+), P_\Phi'(t-)\}$ for $t>0$  and 
$$\bigcup_{\delta>0} \text{\rm cl}^\delta_+(\partial
P_{{ \Phi}}(\R_+))=\bigcup_{t>0}[P'_\Phi(t-), P_\Phi'(\infty)].$$
\end{proof}

\begin{re}\label{re-4.10}  Theorem \ref{thm-1.3}(i) has a nice application in the multifractal analysis of measures on symbolic spaces.
 Let $\mu$ be a fully supported
Borel probability measure
on the one-sided  full shift space $(\Sigma,\sigma)$ over a finite alphabet. Assume in addition that
\begin{equation}
\label{e-claim}
\mu(I_{n+m}(x))\leq C\mu(I_n(x))\mu(I_m(\sigma^n x)),\qquad \forall\; x\in \Sigma,\; n,m\in \N,
\end{equation}
where $C>0$ is a constant and $I_n(y)$ denotes the $n$-th cylinder in $\Sigma$ containing $y$.
Let $\Phi$ be a potential on $\Sigma$ given by $\Phi=\{\log \mu(I_n(x))\}_{n=1}^\infty$. By applying a general multifractal
result of Ben Nasr \cite{Ben94}, Testud \cite{Tes06} obtained that (formulated in our terminologies as)
$$\htop(\sigma, E_\Phi(\alpha))=\inf\{P_\Phi(q)-\alpha q:\; q>0\} \quad \mbox{whenever $\alpha=P'(t)$ for some $t>0$ },$$
provided that $P_\Phi'(t)$ exists at $t$.
However by Proposition \ref{pro-le}(i),  $\Phi$ is asymptotically sub-additive, hence by Theorem
\ref{thm-1.3}(i),
the above variational relation actually holds for any $\alpha=P_\Phi'(t+)$ and $\alpha=P'_\Phi(t-)$ for each $t>0$. Furthermore, the constant $C$ in \eqref{e-claim}
can be replaced by $C_n$, where $(C_n)$ is a sequence of positive numbers satisfying $\lim_{n\to \infty}(1/n)\log C_n=0$ (cf. Remark \ref{re-last}(iii)).

\end{re}

\subsection{Lyapunov spectrum for certain sub-additive potentials on symbolic spaces}
In this subsection, we assume that $(X,T)$ is the one-sided full shift over an finite set  $\{1,2,\ldots,m\}$. That is,
$X=\{1,\ldots,m\}^\N$ endowed with the standard metric $d(x,y)=2^{-n}$ for $x=(x_i)_{i=1}^\infty$ and $(y_i)_{i=1}^\infty$, where $n$
is the largest integer so that $x_j=y_j$ for $1\leq j\leq n$, and  $T$ is the shift map given by
$(x_i)_{i=1}^\infty\mapsto (x_{i+1})_{i=1}^\infty$.

Let $X^*$ denote the collection of finite words over $\{1,\ldots,m\}$, i.e.,
$X^*=\bigcup_{i=1}^\infty \{1,\ldots,m\}^i$.
Assume that $\phi:\; X^*\to [0,\infty)$ is a map (not identically equal to $0$) satisfying the following two assumptions:
\begin{itemize}
\item[(H1)] $\phi(IJ)\leq \phi(I)\phi(J)$ for any $I,J\in X^*$;
\item[(H2)] There exist a sequence of positive integers $(t_n)$ and a sequence of positive numbers $(c_n)$ with
$\lim_{n\to \infty} t_n/n=0=\lim_{n\to \infty} (1/n)\log c_n$, such that for each $I, J\in X^*$ with lengths
$|I|\geq n$, $|J|\geq n$,  there exists $K\in X^*$ with $|K|\leq t_n$ so that
$$
\phi(IKJ)\geq c_n\phi(I)\phi(J).
$$
\end{itemize}

Let $\Phi=(\log \phi_n)$ be a potential on $X$ given by
$\phi_n(x)=\phi(x_1\ldots x_n)$ for $x=(x_i)_{i=1}^\infty$. It is clear that
$\Phi$ is sub-additive.
Denote
$$
P(q):=\limsup_{n\to \infty}\frac{1}{n}\log \sum \phi(I)^q,
$$
where the sum is taken over the set of  $I\in \{1,\ldots,m\}^n$ with $\phi(I)>0$. It is clear that
$P(q)=P(T,q\Phi)$ for $q>0$. Although $\Phi$ is not necessary to be asymptotically additive, we still have the following
rather complete result, as an analogue of our recent work \cite{Fen08} on the norm of
 matrix products.

\begin{thm}
\label{thm-new1}
Let $\Phi$ be given as above. Assume that $P(q)\in \R$ for each $q<0$. Then $\{\alpha\in \R:\; E_\Phi(\alpha)\neq \emptyset\}=\R\cap [a,b]$, where
$a=\lim_{n\to -\infty}P(q)/q$ and $b=\lim_{n\to \infty}P(q)/q$. Furthermore, for $\alpha\in \R\cap [a,b]$,
$$
\htop(T, E(\alpha))=\inf\{P(q)-aq:\; q\in \R\}.
$$
\end{thm}
\begin{proof}[Proof of Theorem \ref{thm-new1}] Take a slight modification of the proof of Theorem 1.1 in \cite{Fen08}.
\end{proof}
We remark that under the condition of the above theorem, we always have $b\in \R$.
However it is possible that $a=-\infty$.

There are some natural maps  $\phi:\; X^*\to [0,\infty)$ which satisfy the assumptions (H1)-(H2).
For example, if $\{M_i\}_{i=1}^m$ is a family of $d\times d$ real matrices so that  there is no trivial proper linear subspace
$V\subset {\Bbb R}^d$ with $M_iV\subseteq V$ for all $i$, then the map $\phi$ defined by $\phi(x_1\ldots x_n)=
\|M_{x_1}\ldots M_{x_n}\|$ satisfies (H1)-(H2) (see \cite{Fen08, FaSl09}).   More generally, the singular value functions
for $M_{x_1}\ldots M_{x_n}$ also satisfy (H1)-(H2) when $\{M_i\}_{i=1}^m$ satisfies further mild irreducibility conditions
(see \cite{FaSl09}).

\section{The multifractal formalism for asymptotically additive potentials}
\label{S-5}
\setcounter{equation}{0}

Let $(X,T)$ be a TDS. Let $k\in \N$ and let $\Phi_i=\{\log \phi_{n,i}\}_{n=1}^\infty$, $\Psi_i=\{\log \psi_{n,i}\}_{n=1}^\infty$
 ($i=1,2,\cdots,k$) be asymptotically additive
potentials on $(X,T)$. Furthermore assume
\begin{equation}\label{assump}
\psi_{n,i}(x)\geq C(1+\delta)^n\qquad(i\in \{1,2,\cdots,k\},\; n\in
\N,\; x\in X)
 \end{equation}
 for some constants $C,\delta>0$. This assumption guarantees that
 $(\Psi_i)_*(\mu)\neq 0$ ($i=1,2,\cdots,k$) for each $\mu\in \M(X,T)$.

For $\a=(a_1,\ldots,a_k)\in \R^k$, denote
\begin{equation}
\label{e-ebfa}
E(\a):=\left\{x\in X:\;\lim_{n\to \infty}\frac{\log \phi_{n,i}(x)}
{\log \psi_{n,i}(x)}=a_i \mbox{ for }1\leq i\leq k\right\}.
\end{equation}
In this section, we shall study the multifractal structure of $E(\a)$.

For $\mu\in \M(X,T)$, the set $G_\mu$ of {\it $\mu$-generic points} is defined by
$$
G_\mu:=\left\{x\in X:\; \frac{1}{n}\sum_{j=0}^{n-1}\delta_{T^jx}\to \mu \mbox{ in the weak* topology as $n\to \infty$}
\right\},
$$
where $\delta_y$ denotes the probability measure whose support is the single point $y$.
Bowen \cite{Bow73} showed that $\htop(T, G_\mu)\leq h_\mu(T)$ for any $\mu\in \M(X,T)$.
\begin{de}
A TDS $(X,T)$ is called to be {\it saturated} if for any $\mu\in \M(X,T)$, we have $G_\mu\neq \emptyset$ and $\htop(T, G_\mu)=h_\mu(T)$.
\end{de}

It was shown independently in \cite{FLP08,PfSu07} that if a TDS satisfies the specification property
(or a weaker form), then it is saturated. The main result in this section is the following.

\begin{thm}
\label{thm-6.1} Let  $(X,T)$ be a TDS and let $\Phi_i, \Psi_i$
($i=1,2,\cdots,k$) be asymptotically additive potentials on $X$
satisfying the assumption (\ref{assump}). Let $\Omega\subset \R^k$
be the range of the following map from $\M(X,T)$ to $\R^k$:
$$\mu\to \left(\frac{(\Phi_1)_*(\mu)}{(\Psi_1)_*(\mu)}, \;\frac{(\Phi_2)_*(\mu)}{(\Psi_2)_*(\mu)}, \cdots,
\;\frac{(\Phi_k)_*(\mu)}{(\Psi_k)_*(\mu)}\right).$$ For
$\a\in \Omega$, write
\begin{equation}\label{e-tt22}
H(\a)=\sup\{h_\mu(T): \; \mu\in \M(X,T),\;
(\Phi_i)_*(\mu)=a_i(\Psi_i)_*(\mu) \mbox{ for } i=1,2,\cdots,k\}.
\end{equation}
Then we have the following properties:
\begin{itemize}
\item[(i)] $\{\a\in\R^k:\; E(\a)\neq \emptyset\}\subseteq \Omega$.
\item[(ii)] If $\htop(T)<\infty$, then we have
$$
\a\in \Omega \Longleftrightarrow \inf\{P_\a({\bf q}):\; {\bf q}\in \R^k\}\neq -\infty\Longleftrightarrow \inf\{P_\a({\bf q}):\; {\bf q}\in \R^k\}\geq 0,
$$
where $P_{\a}({\bf q}):=P\left(T, \sum
\limits_{i=1}^k q_i(\Phi_i-a_i\Psi_i)\right)$.

\item[(iii)] Assume  that $\htop(T)<\infty$ and the entropy map is upper semi-continuous. Then for any
$\a\in \Omega$,
$$H(\a)=\inf_{{\bf q}\in \R^k}P_{\a}({\bf q}),
$$

\item[(iv)] Assume that $(X,T)$ is saturated. Then
$E(\a)\neq \emptyset$ if and only if
$\a\in \Omega$. Furthermore $$\htop
(T,E(\a))=H(\a),\qquad \forall\;
\a\in \Omega.$$

\end{itemize}
 \end{thm}
\bigskip

We emphasize that in parts (i)-(iii) of the above theorem, we do not need to assume that $(X,T)$ is saturated.

\begin{proof}We first prove (i). Assume that $E(\a)\neq \emptyset$ for some $\a=(a_1,\ldots,a_k)\in \R^k$. Take $x\in E(\a)$. Denote $\mu_{x,n}=(1/n)
\sum_{j=1}^{n-1}\delta_{T^jx}$. Then there exists $n_j\uparrow \infty$ so that $\mu_{x,n_j}\to \mu$ for some $\mu\in \M(X,T)$.
Apply Lemma \ref{lem-3.4}(ii) (in which we take $\nu_n=\delta_x$) to obtain
$$
\frac{(\Phi_i)_*(\mu)}{(\Psi_i)_*(\mu)}=\lim_{j\to \infty}\frac{\log \phi_{n_j,i}(x)}{\log \psi_{n_j,i}(x)}= a_i,\quad i=1,\ldots,k.
$$
Hence $\a\in \Omega$. This proves (i).

To show (ii), assume $\htop(T)<\infty$. For $\a=(a_1,\ldots,a_k)\in \R^k$ and $\mu\in \M(X,T)$, we denote
$$
\tau_\a(\mu)=\left( (\Phi_1)_*(\mu)-a_1 (\Psi_1)_*(\mu),\;
\ldots,\;
(\Phi_k)_*(\mu)-a_k (\Psi_k)_*(\mu)\right).
$$
Clearly, $\tau_\a(\mu)\in \R^k$, and
$$\a\in \Omega \Longleftrightarrow \tau_\a(\mu_0)={\bf 0} \mbox{ for some }\mu_0\in \M(X,T).$$

Now  assume $\a\in \Omega$. Then there exists $\mu_0\in \M(X,T)$ such that $\tau_\a(\mu_0)={\bf 0}$.
  Apply Theorem \ref{thm-3.3} to obtain that
  \begin{eqnarray*}
  P_\a({\bf q})&=&P\left(T, \sum
\limits_{i=1}^k q_i(\Phi_i-a_i\Psi_i)\right)\\
&=&\sup\{h_\mu(T)+ \bq\cdot \tau_\a(\mu):\; \mu\in \M(X,T)\}
 \geq h_{\mu_0}(T)\geq 0
 \end{eqnarray*}
  for each ${\bf q}=(q_1,\ldots,q_k)\in \R^k$.

 Conversely,  assume $\a\not\in \Omega$.
write
\begin{equation}
\label{tt11}
A=\{\tau_\a(\mu):\;  \mu\in \M(X,T) \}.
\end{equation}
By Lemma \ref{lem-3.4}(i), $\tau_\a$ is a continuous affine function  on $\M(X,T)$, hence $A$ is a compact convex set in $\R^k$.  $\a\not\in \Omega$ implies ${\bf 0}\not\in A$. Hence there exists a unit vector ${\bf v}\in \R^k$ and $c>0$
such that
$$
{\bf v}\cdot {\bf b}<-c \quad \mbox{ for any }{\bf b}\in A.
$$
By Theorem \ref{thm-3.3},  we have for $t>0$,
\begin{eqnarray*}
P_\a(t{\bf v}) &=& \sup \{ h_\mu(T) +t{\bf v}\cdot \tau_\a(\mu):\; \mu\in \M(X,T)\}\\
&\leq &\sup\{h_\mu(T)-tc:\;\mu\in \M(X,T)\}=\htop(T)-tc.
\end{eqnarray*}
Letting $t\to +\infty$, we obtain   $\inf\{P_\a({\bf q}):\; {\bf q}\in \R^k\}= -\infty$. This finishes the proof of (ii).

Next we prove (iii).
Fix $\a=(a_1,\ldots, a_k)\in \Omega$. Define  $A$ as in \eqref{tt11}.
 Since $\a\in \Omega$, we have ${\bf 0}\in A$. Define $g:\ A\to \R$ by
$$
g({\bf t})=\sup\{h_\mu(T):\; \mu\in \M(X,T),\; \tau_\a(\mu)={\bf t}\}.
$$
It is direct to check that $g$ is concave and upper semi-continuous on $A$. By the definition of $H$ (see \eqref{e-tt22}), we have $H(\a)=g({\bf 0})$.
Define
$$
W({\bf q})=\sup\{g({\bf t})+{\bf q}\cdot {\bf t}:\; {\bf t}\in A\},\qquad \forall \; {\bf q}\in \R^k.
$$
Then by Corollary \ref{cor-0330}(ii), we have $g({\bf t})=\inf\{W({\bf q})-{\bf q}\cdot {\bf t}:\; {\bf q}\in \R^k\}$ for
all ${\bf t}\in A$. In particular,
\begin{equation}
\label{e-0331}
H(\a)=g({\bf 0})=\inf\{W({\bf q}):\; {\bf q}\in \R^k\}.
\end{equation}
However, by Theorem \ref{thm-3.3} and the definition of $W$, we have $$W({\bf q})=P\left(T, \sum
\limits_{i=1}^k q_i(\Phi_i-a_i\Psi_i)\right)=:P_\a({\bf q}).$$ Hence (\ref{e-0331}) implies
$H(\a)=\inf_{{\bf q}\in \R^k} P_\a({\bf q})$. This finishes the proof (iii).

In the end we prove (iv).  We divide this proof into four steps.

\noindent{\sl Step 1.}  {\em For $\a\in \Omega$, we
have $E(\a)\supseteq G_\mu\neq \emptyset$ for each
$\mu\in \M(X,T)$ with  $(\Phi_i)_*(\mu)=a_i(\Psi_i)_*(\mu)$
($i=1,2,\cdots,k$)}. To see this, let $x\in G_\mu$. By Lemma
\ref{lem-3.4}(ii) (in which we take  $\nu_n=\delta_x$), we have
$\lim_{n\to \infty} (1/n)\log \phi_{n,i}(x)=(\Phi_i)_*(\mu)$ and
$\lim_{n\to \infty} (1/n)\log \psi_{n,i}(x)=(\Psi_i)_*(\mu)$. It follows
that   $x\in E(\a)$. Hence
$E(\a)\supseteq G_\mu$.

\medskip
\noindent{\sl Step 2.} {\em Let $\a\in \R^k$ so that
$E(\a)\neq \emptyset$. Then for each $x\in
E(\a)$ and $\mu\in V(x)$ (here $V(x)$ denotes the set of limit points of $\mu_{x,n}=(1/n)\sum_{j=0}^{n-1}\delta_{T^jx}$), we have
$(\Phi_i)_*(\mu)/(\Psi_i)_*(\mu)=a_i$ for $i=1,2,\cdots,k$.} To
show this, take such  $x$ and $\mu$.  Then there exists a
subsequence sequence $n_\ell$ of natural numbers such that
$\lim_{\ell\to \infty} \mu_{n_\ell,x}=\mu$. By Lemma
\ref{lem-3.4}(ii) again (in which we take  $\nu_n=\delta_x$), we
have
$$
\lim_{\ell\to \infty}\frac{1}{n_\ell} \log \phi_{n_\ell,i}(x)
=(\Phi_i)_*(\mu) \mbox { and } \lim_{\ell\to \infty}\frac{1}{n_\ell}
 \log \psi_{n_\ell,i}(x)=(\Psi_i)_*(\mu)\qquad (i=1,2,\cdots,k).
$$
Since $x\in E(\a)$, we have
$$
\lim_{n\to \infty}\frac{ \log \phi_{n,i}(x)}{\log \psi_{n,i}(x)} =a_i
\quad (i=1,2,\cdots,k).
$$
It follows that $(\Phi_i)_*(\mu)/(\Psi_i)_*(\mu)=a_i$ for
$i=1,2,\cdots,k$.

\medskip
\noindent{\sl Step 3.} {\em For $\a\in \Omega$, we
have $\htop (T, E(\a))\geq H(\a)$}.  To
see it, let $\mu\in \M(X,T)$ so that
$(\Phi_i)_*(\mu)=a_i(\Psi_i)*(\mu)$ ($i=1,2,\cdots,k$). By step 1,
$E(\a)\supseteq G_\mu$ and hence $\htop (T,
E(\a))\geq \htop(T, G_\mu)=h_\mu(T)$. This proves the
inequality $\htop (T, E(\a))\geq
H(\a)$.

\medskip
\noindent{\sl Step 4.} {\em For $\a\in \Omega$, we
have $\htop (T, E(\a))\leq H(\a)$}. By
step 2, for each $x\in E(\a)$ and $\mu\in V(x)$, we
have $(\Phi_i)_*(\mu)/(\Psi_i)_*(\mu)=a_i$ for $i=1,2,\cdots,k$ and
hence $h_\mu(T)\leq H(\a)$. It follows that $$
E(\a)\subseteq \{x\in X:\; \exists \; \mu\in V(x)
\mbox{ with  }h_\mu(T)\leq H(\a)\}.
$$
By Lemma \ref{Bow}, we have $\htop(T, E(\a))\leq
H(\a)$. This finishes the proof of (iv).
\end{proof}

We remark that (iii) of Theorem \ref{thm-6.1} can be proved alternatively  by applying Proposition 3.15 in \cite{FeSh08}.
\medskip
\begin{proof}[Proof of Theorem \ref{thm-1.4}]
Except the second part in (iii), all the statements listed in
Theorem \ref{thm-1.4} follow from Theorem \ref{thm-6.1} (in which we
take $k=1$ and $\psi_n(x)\equiv 1$). In the following,  we prove the
second part in Theorem \ref{thm-1.4}(iii): under the assumptions
that $\htop(T)<\infty$ and the entropy map is upper semi-continuous,
for any  $\alpha\in \Gamma:=\bigcup_{t\in
\R}\{P^\prime_\Phi(t-),P^\prime_\Phi(t+)\}\cup
\{P_\Phi^\prime(\pm\infty)\}$,
 we have $E_\Phi(\alpha)\neq \emptyset$ and $$
\htop(T,E_\Phi(\alpha))=\inf\{P_\Phi(q)-\alpha q:\; q\in \R\}.
$$
According to Corollary \ref{cor-3.8}, it suffices to show that if
$\alpha\in \Gamma$, then $E_\Phi(\alpha)\neq \emptyset$ and $
\htop(T,E_\Phi(\alpha))\geq \inf\{P_\Phi(q)-\alpha q:\; q\in \R\}. $
For this purpose,  we will show the following claim:

\noindent {\bf Claim.} {\it For each $\alpha\in \Gamma$, there exists
an ergodic measure $\nu$ such that $\Phi_*(\nu)=\alpha$ and
$h_\nu(T)\geq \inf\{P_\Phi(q)-\alpha q:\; q\in \R\}$}.

The claim will imply   that $\nu(E_\Phi(\alpha))=1$ (by  Kingman's
sub-additive ergodic theorem), and
 by Proposition \ref{pro-2.1}(3),   $\htop(T, E_\Phi(\alpha))\geq h_\nu(T)\geq \inf\{P_\Phi(q)-\alpha q:\; q\in \R\}$. In the following we prove the claim
 in a way similar to the proof
of Theorem \ref{thm-1.3-h}.

First we consider the case  $\alpha\in
\{P^\prime_\Phi(t-),P^\prime_\Phi(t+)\}$ for some $t\in \R$.  Fix
$t$ and denote
$$\mathcal{I}({ \Phi},{ t})=\{\mu\in \M(X,T):\; P_\Phi(t)=h_\mu(T)+t\Phi_*(\mu)\}.$$
 By Theorem
\ref{lem-mfsa1-1}(ii), $\mathcal{I}({ \Phi},{ t})$ is a non-empty
compact convex subset of $\mathcal{M}(X,T)$. Furthermore, for any
$b\in \R$ the set
$$\mathcal{I}_{b}({ \Phi},{ t}):=
\{\nu\in \mathcal{I}({ \Phi},{ t}):\ { \Phi}_*(\nu)=b\}$$ is compact
and convex (may be empty). The convexity of $\mathcal{I}_{b}({
\Phi},{ t})$ is clear. To show the compactness, assume that
$\{\nu_n\}\subset \mathcal{I}_{b}({ \Phi},{ t})$ and $\nu_n$
converges to $\nu$ in $\mathcal{M}(X,T)$. Then by the upper
semi-continuity of $h_{(\cdot)}(T)$ and the continuity of
$\Phi_*(\cdot)$, we have
$$h_\nu(T)+{ t}{ \Phi}_*(\nu)\geq \lim_{n\to
\infty}h_{\nu_n}(T)+{t} { \Phi}_*(\nu_n)=P_{{ \Phi}}({ t}).$$ By
Theorem \ref{thm-3.3}, $\nu\in \mathcal{I}({ \Phi},{ t})$ and
furthermore, $h_\nu(T)=h_{\nu_n}(T)=P_{{ \Phi}}({t})-{ t} b$ and
$ { \Phi}_*(\nu)={ \Phi}_*(\nu_n)=b$. This is,
 $\nu\in \mathcal{I}_{b}({ \Phi},{ t})$.
Hence $\mathcal{I}_{b}({ \Phi},{ t})$ is compact.

Since $\alpha\in \{ P'_{\Phi}(t-),P'_{\Phi}(t+)\}\subset \partial
P_\Phi(t)$, by  Theorem \ref{lem-mfsa1-1}, $\mathcal{I}_{\alpha}({
\Phi},{ t})$ is non-empty. We are going to show further that
$\mathcal{I}_{\alpha}({\Phi},{t})$ contains at least one  ergodic
measure. Since $\mathcal{I}_{\alpha}({\Phi},{ t})$
 is a non-empty
compact convex subset of $\mathcal{M}(X,T)$, by the Krein-Milman
theorem, it contains at least one extreme point, denoted by $\nu$.
Let $\nu=p\nu_1+(1-p)\nu_2$ for some $0<p<1$ and $\nu_1,\nu_2\in
\mathcal{M}(X,T)$. We will show that $\nu_1=\nu_2$, which implies
that $\nu$ is ergodic. To see that $\nu_1=\nu_2$, note that
\begin{align*}
P_{{ \Phi}}({ t})&=h_\nu(T)+{ t} { \Phi}_*(\nu) \\
&=p(h_{\nu_1}(T)+{ t} { \Phi}_*(\nu_1))+(1-p)(h_{\nu_2}(T)+{ t} {
\Phi}_*(\nu_2)).
\end{align*}
By Theorem \ref{thm-3.3}, $\nu_1,\nu_2\in \mathcal{I}({ \Phi},{
t})$. By Theorem \ref{lem-mfsa1-1}, ${ \Phi}_*(\nu_1),{
\Phi}_*(\nu_2)\in
\partial P_{{ \Phi}}({ t})$. Moreover,
note that $\alpha={ \Phi}_*(\nu)=p{ \Phi}_*(\nu_1) +(1-p){
\Phi}_*(\nu_2)$, we have ${ \Phi}_*(\nu_1)={ \Phi}_*(\nu_2)=\alpha$
since $\alpha$ is an extreme point of $\partial P_{{ \Phi}}({ t})$
(noting  that $\partial P_{{ \Phi}}({
t})=[P'_{\Phi}(t-),P'_{\Phi}(t+)]$). That is, $\nu_1,\nu_2\in
\mathcal{I}_{\alpha}({ \Phi},{ t})$. Since $\nu$ is an extreme point
of $\mathcal{I}_{\alpha}({ \Phi},{ t})$, we have $\nu_1=\nu_2=\nu$.
Therefore,  $\nu$ is ergodic. Since $\nu\in
\mathcal{I}_{\alpha}({\Phi},{ t})$, we have $\Phi_*(\nu)=\alpha$,
and $h_{\nu}(T)=P_\Phi(t)-\alpha t\geq \inf \{P_\Phi(q)-\alpha q:\;
q\in \R\}$. This proves the claim in the case that $\alpha\in \{
P'_{\Phi}(t-),P'_{\Phi}(t+)\}\subset \partial P_\Phi(t)$ for some
$t\in \R$.

Next, we consider the case  $\alpha\in \{P'_\Phi(\pm \infty)\}$.
First assume that $\alpha=P'_\Phi(+\infty)$. By Theorems
\ref{thm-1.1} and \ref{thm-1.2},
$\alpha=\overline{\beta}(\Phi)=\max\{\Phi_*(\mu):\; \mu\in
\M(X,T)\}$. By the convexity of $P_\Phi(\cdot)$, there exists a
sequence $(t_j)\uparrow +\infty$, such that $P'_\Phi(t_j):=\alpha_j$
exists and $\alpha_j\uparrow\alpha$ when $j\to \infty$.  As  proved
in last paragraph,
 for each $j\in \mathbb{N}$, there exists $\mu_j\in
\mathcal{M}(X,T)$ such that ${ \Phi}_*(\mu_j)=\alpha_j$  and
$h_{\mu_j}(T)= P_{{ \Phi}}({ t}_j)-\alpha_j { t}_j$. Extract a
subsequence if necessary so that  $\lim_{j\rightarrow
\infty}\mu_j=\mu$ for some $\mu\in \mathcal{M}(X,T)$. Then
${\Phi}_*(\mu)=\lim_{j\to \infty}{\Phi}_*(\mu_j)=\alpha$ and
\begin{align*}
h_\mu(T)&\ge \limsup_{j\rightarrow \infty} h_{\mu_j}(T)=
\limsup_{j\rightarrow \infty} (P_{{\Phi}}({
t}_j)-\alpha_j {  t}_j)\\
&\ge \limsup_{j\rightarrow \infty} (P_{{ \Phi}}({ t}_j)-\alpha {
t}_j)\geq \inf_{{ q}\in \mathbb{R}}\{ P_{{ \Phi}}({ q})-\alpha
{q}\},
\end{align*}
where we use the facts $\alpha_j\leq \alpha$ and $t_j>0$ for the
second inequality. Now let $\mu=\int \theta \; dm(\theta)$ be the
ergodic decomposition of $\mu\in \M(X,T)$.   By Proposition
\ref{pro-1195}(3), $\int \Phi_*(\theta)\;
dm(\theta)=\Phi_*(\mu)=\alpha$.  By the way, we also have  $\int
h_\theta(T)\; dm(\theta)=h_{\mu}(T)$ (cf. \cite{Wal-book}).  Note
that  $\Phi_*(\theta)\leq \alpha$ for any $\theta$. Hence there
exists an ergodic measure $\nu$ such that $\Phi_*(\nu)=\alpha$ and
$h_\nu(T)\geq h_\mu(T)\geq \inf_{{ q}\in \mathbb{R}}\{ P_{{ \Phi}}({
q})-\alpha {q}\}$, as desired. In the end, assume
$\alpha=P'_\Phi(-\infty)$. Then $-\alpha=P'_{-\Phi}(+\infty)$. Since
$-\Phi$ is also asymptotically additive, there exists an ergodic
measure $\eta$ such that $(-\Phi)_*(\eta)=-\alpha$, i.e.,
$\Phi_*(\eta)=\alpha$,  and  $$ h_\eta(T)\geq \inf_{ q\in \R}\{
P_{{- \Phi}}({ q})-(-\alpha) {q}\}=\inf_{q\in \R}\{ P_{{ \Phi}}({
-q})+\alpha {q}\}=\inf_{q\in \R}\{ P_{{ \Phi}}({ q})-\alpha {q}\}.$$
This finishes the proof of the claim, and also the proof of Theorem
\ref{thm-1.4}.
\end{proof}

\section{Examples}
\label{S-6}
\setcounter{equation}{0}
In this section, we give some  examples regarding  Lyapunov spectra on TDS's on which the entropy map is not
upper semi-continuous. The multifractal behaviors in this case are rather irregular and complicated.
These examples also show that the conditions and results in Theorems \ref{thm-1.1}, \ref{thm-1.2} and \ref{thm-1.2-h} are optimal.

\begin{ex}\label{re-1.1}
There exist a TDS $(X,T)$ with $\htop(T)<\infty$ and an additive
potential $\Phi=\{\log \phi_n\}_{n=1}^\infty$ on $X$ such that $\left(\lim_{t\to
0+}P^\prime_\Phi(t-),P^\prime_\Phi(\infty)\right)=\emptyset$ and for
$\alpha=P'_{\Phi}(\infty)$,
$$\htop(T,E_{\Phi}(\alpha))=\sup\{h_\mu(T): \mu\in
\mathcal{M}(X,T),\Phi_*(\mu)=\alpha\}<\inf_{q>0}\{P_{\Phi}(q)-\alpha
q\}.$$
\end{ex}
\begin{proof}[Construction]
 According to Krieger \cite{Kre72}, for each $i\in \N$, we can  construct a Cantor set  $X_i\subseteq [0,\frac{1}{i}]\times
\{\frac{1}{i}\}$ and a continuous transformation  $T_i:\; X_i\rightarrow X_i$ such that   $(X_i,T_i)$ is
uniquely ergodic (i.e., $\M(X_i, T_i)$ consists of a singleton) and
$h_{\text{top}}(T_i)=1$. Then we let  $X=\bigcup_{i=1}^\infty X_i\cup
\{(0,0)\}$ and define  $T:\; X\rightarrow X$ by
$$T(x)=\begin{cases} T_i(x) \, &\text{ if }x\in X_i,\\
x \, &\text{ if }x=(0,0).\end{cases}$$
It is easy to check that $(X,T)$ is a TDS. Define a function $g:\; X\to \R$ by
$$g(x)=\begin{cases} 1-1/i \, &\text{ if }x\in X_i,\\
1 \, &\text{ if }x=(0,0).\end{cases}$$ Let
$\phi_n(x)=\exp \left( \sum \limits_{j=0}^{n-1}g(T^jx) \right)$ for $n\in \mathbb{N}$ and $x\in
X$. Then $\Phi=\{\log \phi_n\}_{n=1}^\infty$ is an additive potential on $X$.

For $i\in \N$, let $\mu_i$ denote the unique element in $\M(X_i,T_i)$. Let $\mu_0$ be
the Dirac measure $\delta_{(0,0)}$ at the point $(0,0)$. Then
$\mathcal{E}(X,T)=\{ \mu_i:\, i=0,1,\cdots\}$ and thus
$$\mathcal{M}(X,T)=\left\{\sum_{i=0}^\infty \lambda_i\mu_i:\, \lambda_i\ge 0 \text{ and }
\sum_{i=0}^\infty \lambda_i=1\right\}.$$
By Theorem \ref{thm-3.3},  we have
\begin{align*}
P_{\Phi}(q)&=\sup\{ h_\mu(T)+q\Phi_*(\mu):\mu\in \mathcal{M}(X,T)\}\\
&=\sup\left\{ h_{\sum_{i=0}^\infty \lambda_i \mu_i}(T)+q\sum_{i=0}^\infty \lambda_i \Phi_*(\mu_i):\lambda_i\ge 0
\text{ and }
\sum_{i=0}^\infty \lambda_i=1 \right\}\\
&=\sup\left\{ \sum_{i=0}^\infty \lambda_i \left(  h_{\mu_i}(T)+q\Phi_*(\mu_i)\right) :\lambda_i\ge 0
 \text{ and }
\sum_{i=0}^\infty \lambda_i=1 \right \}\\
&=\sup\left\{ q\lambda_0+\sum_{i=1}^\infty \lambda_i \left(1+q(1-1/i)) \right) :\lambda_i\ge 0
 \text{ and }
\sum_{i=0}^\infty \lambda_i=1  \right\}\\
&=\max\left\{ q, \; \sup_{i\in \mathbb{N}} \{1+q(1-{1}/{i})
\}\right\}=q+1 \text{ for }q> 0.
\end{align*}
Hence $P'_{\Phi}(q)=1\, \text{ for }q>0 \text{ and }
P'_{\Phi}(\infty)=1$. Thus $\left(\lim_{t\to
0+}P^\prime_\Phi(t-),P^\prime_\Phi(\infty)\right)=\emptyset$.  For
$\alpha=P'_{\Phi}(\infty)=1$, one has $E_{\Phi}(\alpha)=\{(0,0)\}$. Hence
$$0=h_{\text{top}}(T,E_{\Phi}(\alpha))=\sup\{ h_\mu(T): \mu\in
\M(X,T), \Phi_*(\mu)=\alpha\}<\inf_{q>0}\{P_{\Phi}(q)-\alpha q\}
=1,$$ as desired.
\end{proof}

\begin{ex}\label{re-1.2}
There exist a TDS $(X,T)$ with $\htop(T)<\infty$, an additive
potential $\Phi=\{\log \phi_n\}_{n=1}^\infty$ on $X$ such that for each $\alpha\in
[\underline{\beta}(\Phi),\overline{\beta}(\Phi)]$,
$$\htop(T,E_{\Phi}(\alpha))< \inf_{q\in \mathbb{R}}\{P_{\Phi}(q)-\alpha
q\},$$
where $\underline{\beta}(\Phi):=\lim_{n\to \infty}\frac{1}{n} \inf_{x\in X} \log \phi_n(x)$.
\end{ex}

\begin{proof}
[Construction]
Similar to the construction  in Example \ref{re-1.1}, we construct
Cantor sets  $X_i\subseteq [0,\frac{1}{|i|+1}]\times
\{\frac{i}{|i|+1}\}$ ($i\in \mathbb{Z}$) and  continuous
transformations $T_i:X_i\rightarrow X_i$ such that $(X_i,T_i)$ is
uniquely ergodic and $h_{\text{top}}(T_i)=\frac{|i|}{|i|+1}$.
 Then let
$X=\bigcup_{i\in \mathbb{Z}} X_i\cup \{(0,1)\}\cup\{(0,-1)\}$ and
define $T:X\rightarrow X$ by
$$T(x)=\begin{cases} T_i(x) \, &\text{ if }x\in X_i,\\
x \, &\text{ if }x=(0,1)\text{ or }(0,-1).\end{cases}$$
It is clear that  $(X,T)$ is a TDS. Define a continuous function $h$
on $X$ by
$$g(x)=\begin{cases} \frac{i}{|i|+1} \, &\text{ if }x\in X_{i},\\
1 \, &\text{ if }x=(0,1),\\
-1 \, &\text{ if }x=(0,-1).
\end{cases}$$
Let $\phi_n(x)=\exp \left(\sum \limits_{j=0}^{n-1}g(T^jx)\right)$
for $n\in \mathbb{N}$ and $x\in X$. Then $\Phi=\{\log \phi_n\}$ is
an additive potential on $X$ with
$[\underline{\beta}(\Phi),\overline{\beta}(\Phi)]=[-1,1]$.
Similarly, it is not hard to verify that
$$P_{\Phi}(q)=\max\left\{ q,-q,\;\sup_{i\in {\Bbb Z}}\left\{\frac{|i|}{|i|+1}+\frac{i}{|i|+1}q\right\}\right\}
=1+|q|.$$
Hence $P'_{\Phi}(\infty)=P_{\Phi}'(0+)=1,\,
P_{\Phi}'(0-)=P_{\Phi}'(-\infty)=-1 $ and
$$P'_{\Phi}(q)=\begin{cases} 1 \, &\text{ if }q> 0,\\
-1 \, &\text{ if }q<0.
\end{cases}
$$
It is easy to see that $E_\Phi(\alpha)\neq \emptyset$
if and only if $\alpha\in \{\frac{i}{|i|+1}:\; i\in {\Bbb Z}\}\cup\{1,-1\}$.
Furthermore
$$
E_\Phi(\alpha)=\left\{\begin{array}{ll}
X_i &\mbox{ if } \alpha=\frac{i}{|i|+1} \mbox{ for some }i\in {\Bbb Z},\\
\{(0,1)\} &\mbox{ if }\alpha=1,\\
\{(0,-1)\} &\mbox{ if }\alpha=-1.
\end{array}
\right.
$$
 Hence for $\alpha\in
[\underline{\beta}(\Phi),\overline{\beta}(\Phi)]=[-1,1]$,
$$h_{\text{top}}(T,E_{\Phi}(\alpha))<1=\inf_{q\in \mathbb{R}}\{P_{\Phi}(q)-\alpha
q\},$$
as desired. Keep in mind that  $\{\Phi_*(\mu):\; \mu\in \M(X,T)\}=[-1,1]$ by Lemma \ref{lem-mf-sa}.
\end{proof}
\begin{ex}\label{re-1.3}
There exist a TDS $(X,T)$ with $\htop(T)<\infty$ and two additive
potential $\Phi_i=\{\log \phi_{n,i}\}_{n=1}^\infty$ ($i=1,2)$ on $X$
 such that $\partial P_{\bf
\Phi}(\mathbb{R}^2_+)$ is one-dimensional set and for any ${\bf
a}\in
\partial P_{\bf \Phi}(\mathbb{R}^2_+)$, where ${\bf
\Phi}=(\Phi_1,\Phi_2)$,
$$\sup\{h_\mu(T): \mu\in
\mathcal{M}(X,T),{\bf \Phi}_*(\mu)={\bf a}\}< \inf_{{\bf q}\in
\mathbb{R}_+^2}\{P_{\bf \Phi}({\bf q})-{\bf a}\cdot {\bf q}\}.$$
\end{ex}
\begin{proof}[Construction]
Similar to the previous two examples, we construct a Cantor
set $X_i\subseteq [0,\frac{1}{|i|+1}]\times \{\frac{i}{|i|+1}\}$ and
a continuous transformation  $T_i:\; X_i\rightarrow X_i$ such that
 $(X_i,T_i)$ is uniquely ergodic  and $h_{\text{top}}(T_i)=1$.  Then let
$X=\bigcup_{i\in \mathbb{Z}} X_i\cup \{(0,1)\}\cup\{(0,-1)\}$ and
define $T:X\rightarrow X$ by
$$T(x)=\begin{cases} T_i(x) \, &\text{ if }x\in X_i,\\
x \, &\text{ if }x=(0,1)\text{ or }(0,-1).\end{cases}$$ It is clear
that  $(X,T)$ is a TDS. Define two continuous function $g_1,g_2$ on
$X$ by
$$g_1(x)=\begin{cases} \frac{i}{i+1} \, &\text{ if }x\in X_{i},\; i\ge 0,\\
1 \, &\text{ if }x=(0,1),\\
\frac{2|i|}{|i|+1} \, &\text{ if }x\in X_i,\;i<0,\\
2 \, &\text{ if }x=(0,-1).
\end{cases}
\ \ \ \text{ and } \ \ \
g_2(x)=\begin{cases} 0 \, &\text{ if }x\in X_{i},\; i\ge 0,\\
0 \, &\text{ if }x=(0,1),\\
\frac{-|i|}{|i|+1} \, &\text{ if }x\in X_i,\;i<0,\\
-1 \, &\text{ if }x=(0,-1).
\end{cases}
$$
Set $\phi_{n,i}(x)=\exp \left(\sum
\limits_{j=0}^{n-1}g_i(T^jx)\right)$ for $i=1,2$, $n\in \mathbb{N}$
and $x\in X$. Then $\Phi_i=\{\log \phi_{n,i}\}_{n=1}^\infty$, $i=1,2$, are two
additive potentials on $X$ with $\overline{\beta}(\Phi_1)=2$,
$\overline{\beta}(\Phi_2)=0$.

For $i\in \mathbb{Z}$, let $\mu_i$ denote the unique element in
$\M(X_i,T_i)$. Let $\mu_{\infty}$ be the Dirac measure
$\delta_{(0,1)}$ at the point $(0,1)$.  Let $\mu_{-\infty}$ be the
Dirac measure $\delta_{(0,-1)}$ at the point $(0,-1)$. For simplify,
write $\overline{\mathbb{Z}}=\mathbb{Z}\cup \{\pm \infty\}$. Then $\mathcal{E}(X,T)=\{ \mu_i:\, i\in
\overline{\mathbb{Z}}\}$.
A direct calculation by applying  Theorem
\ref{thm-3.3} yields that  for ${\bf q}=(q_1,q_2)\in \mathbb{R}_+^2$,
\begin{align*}
P_{\bf \Phi}({\bf q})=\max\{ 1+q_1,1+2q_1-q_2\}.
\end{align*}
Hence $$P'_{\bf \Phi}({\bf q})=\begin{cases} (1,0) \, &\text{ if }{\bf
q}\in \mathbb{R}_+^2 \text{ with } q_1<q_2\\
(2,-1) \, &\text{ if }{\bf q}\in \mathbb{R}_+^2 \text{ with }
q_1>q_2 \end{cases}   $$
and $\partial P_{\bf
\Phi}((q,q))=\text{conv}(\{(1,0),(2,-1)\})$ for $q>0$. Thus
$\partial P_{\bf
\Phi}(\mathbb{R}^2_+)=\text{conv}(\{(1,0),(2,-1)\})$ is one
dimensional.

Recall that $A:=\{ {\bf \Phi}_*(\mu):\mu\in \mathcal{M}(X,T)\}$.
Clearly, $A=\text{conv}(\{ (0,0), (1,0), (2,-1)\})$ is
a two-dimensional set and $\partial P_{\bf \Phi}(\mathbb{R}^2_+)$ is just one edge in the convex set $A$.

For ${\bf a}\in \partial P_{\bf \Phi}(\mathbb{R}^2_+)$, there exists
unique $t\in [0,1]$ with ${\bf a}=t(1,0)+(1-t)(2,-1)$. It is not
hard to see that for $\mu\in \mathcal{M}(X,T)$, ${\bf
\Phi}_*(\mu)={\bf a}$ if and only if
$\mu=t\mu_{\infty}+(1-t)\mu_{-\infty}$.  Hence
\begin{align*}
&\sup\{ h_\mu(T): \mu\in
\M(X,T),{\bf \Phi}_*(\mu)={\bf a}\}=h_{t\mu_\infty+(1-t)\mu_{-\infty}}(T)=0\\
&\hskip0.5cm <1=\inf_{{\bf q}\in \mathbb{R}^2_+}\{\max\{1+q_1,1+2q_1-q_2\}-\left( tq_1+(1-t)(2q_1-q_2) \right)\}\\
&\hskip0.5cm=\inf_{{\bf q}\in \mathbb{R}^2_+}\{P_{\bf \Phi}({\bf q})-{\bf a}\cdot {\bf q}\},
\end{align*}
as desired.
\end{proof}

\appendix
\section{Properties and examples of Asymptotical sub-additive potentials}
\label{A}
\setcounter{equation}{0}

In this appendix, we give some properties and examples of
asymptotically sub-additive (resp. asymptotically additive)
potentials. Let $(X,T)$ be a TDS and let $\Phi=\{\log
\phi_n\}_{n=1}^\infty$ be an asymptotically sub-additive potential
on a TDS $(X,T)$. Let $\lambda_\Phi$ and $\Phi_*$ be defined as in
(\ref{e-1.1})-(\ref{e-1.2}).

\begin{pro}
\label{pro-1195}
 Let $\mu\in \mathcal{M}(X,T)$. Then we have the following properties.
\begin{enumerate}
\item   The limit $\Phi_*(\mu)=\lim\limits_{n\rightarrow \infty} \frac{1}{n} \int \log \phi_n(x)\ d
\mu(x)$  exists (which may take value $-\infty$). Furthermore
$\lambda_{\Phi}(x)$ exists  for $\mu$-a.e.~$x\in X$, and $\int
\lambda_{\Phi}(x) ~d \mu(x)=\Phi_*(\mu)$. In particular,  when
$\mu\in \E(X,T)$, $\lambda_{\Phi}(x)=\Phi_*(\mu)$ for $\mu$-a.e.
$x\in X$.

\item The map $\Phi_*:\mathcal{M}(X,T)\rightarrow \mathbb{R}\cup
\{-\infty\}$ is upper semi-continuous and there is $C\in \mathbb{R}$
such that  for all $\mu\in \mathcal{M}(X,T)$, $\lambda_\Phi(x)\le C$
$\mu$-a.e and $\Phi_*(\mu)\le C$.

\item Let $\mu=\int_{\Omega} \theta \ d m(\theta)$ be the ergodic
decomposition of $\mu\in \mathcal{M}(X,T)$. Then
$\Phi_*(\mu)=\int_{\Omega} \Phi_*(\theta) ~d m(\theta)$.

\end{enumerate}
\end{pro}
\begin{proof} In the case that $\Phi$ is  sub-additive, statement (1) comes exactly  from Kingman's sub-additive ergodic theorem (cf. \cite{Wal-book}, p. 231).  We shall show that it remains valid when $\Phi$ is asymptotically sub-additive. Fix such a $\Phi$. For $\epsilon>0$, by definition, there exist a sub-additive
potential $\Psi=\{\log \psi_n\}_{n=1}^\infty$ and an integer $n_0$
such that $|\log \phi_n(x)-\log \psi_n(x)|\leq n\epsilon$ for any
$n\geq n_0$ and $x\in X$. Hence
\begin{eqnarray*}
\limsup_{n\rightarrow \infty} \frac{1}{n}\int \log \phi_n(x) ~d
\mu(x)&\le & \lim_{n\rightarrow \infty}\frac{1}{n}\int \log
\psi_n(x) ~d
\mu(x)+\epsilon\\
&\le & \liminf_{n\rightarrow \infty}\frac{1}{n}\int \log \phi_n(x)
~d \mu(x)+2\epsilon.
\end{eqnarray*}
Since the above inequalities hold for any $\epsilon>0$, the limit
for defining $\Phi_*(\mu)$ exists. Similarly, we have the
inequalities
$$\limsup_{n\rightarrow \infty} \frac{1}{n} \log \phi_n(x) \le \lim_{n\rightarrow \infty}\frac{1}{n} \log \psi_n(x) +\epsilon \le \liminf_{n\rightarrow
\infty}\frac{1}{n} \log \phi_n(x) +2\epsilon
$$
 for $\mu$-a.e. $x$, from which we derive that $\lambda_\Phi(x)$ exists $\mu$-a.e and  $\int \lambda_\Phi(x)\;d\mu(x)=\Phi_*(\mu)$. Furthermore, $\lambda_\Phi(x)=\Phi_*(\mu)$ $\mu$-a.e.\! when $\mu$ is ergodic. This proves (1).

To see that $\Phi_*$ is upper semi-continuous, let $\epsilon>0$ and
$\Psi$ be given  as in the above paragraph.  Suppose that
$\{\mu_i\}$ is a sequence in $\mathcal{M}(X,T)$ which converges to
$\mu$ in the weak$^*$ topology.  Then for any $n\geq n_0$ and $R\in
\R$,
\begin{align*}
\limsup_{i\rightarrow \infty} \Phi_*(\mu_i)& \le
\limsup_{i\rightarrow \infty} \Psi_*(\mu_i)+\epsilon \le
\limsup_{i\rightarrow \infty} \frac{1}{n}\int \log \psi_n(x) ~d
\mu_i(x)+\epsilon\\
&\le\limsup_{i\rightarrow \infty}\frac{1}{n}\int \max\left\{\log
\psi_n(x), R\right\} ~d
\mu_i(x)+\epsilon\\
&=\frac{1}{n}\int \max\left\{\log \psi_n(x), R\right\} ~d
\mu(x)+\epsilon.\\
\end{align*}
Taking $R\to -\infty$ to obtain
\begin{align*}
\limsup_{i\rightarrow \infty} \Phi_*(\mu_i) &\leq \frac{1}{n}\int
\log  \psi_n(x) ~d \mu(x)+\epsilon\leq  \frac{1}{n}\int \log
\phi_n(x) ~d \mu(x)+2\epsilon.
\end{align*}
Letting $n\to \infty$, we have $\limsup_{i\to
\infty}\Phi_*(\mu_i)\leq \Phi_*(\mu)$. This proves the upper
semi-continuity of $\Phi_*$. To give an upper bound for
$\lambda_\Phi$ and $\Phi_*$, let $D=\max_{x\in X}\psi_{n_0}(x)$.
Then $\log \psi_{kn_0}(x)\leq k\log D$ by the subadditivity. Hence
for $\mu$-a.e $x$,
$$\lambda_\Phi(x)\leq \epsilon+\limsup_{k\to \infty}\frac{1}{kn_0}\log \psi_{kn_0}(x)\leq \epsilon+ (\log D)/n_0.$$  Take integration with respect to $\mu$ to get  $\Phi_*(\mu)\leq \epsilon+(\log D)/n_0$.

To prove (3), we first assume that $\Phi$ is sub-additive.  Let
$\mu=\int_{\Omega} \theta \ d m(\theta)$ be the ergodic
decomposition of $\mu\in \mathcal{M}(X,T)$.  Let $C_1=\max_{x\in
X}|\log \phi_1(x)|$. Then
\begin{align}\label{mul-eq}
\frac{1}{n}\int \log \phi_n(x) ~d \theta(x)\le C_1 \quad \text{ for
all }\theta\in \Omega, \, n\in \mathbb{N}.
\end{align}
Define $h_k(\theta)=\frac{1}{2^k} \int\log \phi_{2^k}(x) ~d
\theta(x)$ for $\theta\in \mathcal{M}(X,T)$ and $k\in \mathbb{N}$.
Since  $\Phi$ is sub-additive and $\theta$ is invariant, we have
$C_1\ge h_1(\theta)\ge h_2(\theta)\ge \cdots$ and
$h_k(\theta)\searrow \Phi_*(\theta)$. By \eqref{mul-eq}, we have
\begin{align*}
\Phi_*(\mu)&=\lim \limits_{n\rightarrow \infty} \frac{1}{n}\int \log
\phi_n(x) ~d \mu(x)= \lim \limits_{k\rightarrow \infty}
\int_{\Omega}\frac{1}{2^k} \int \log \phi_{2^k}(x) ~d \theta(x)
d m(\theta)\\
&=\lim \limits_{k\rightarrow \infty} \int h_k(\theta) ~d m(\theta)
=\int_{\Omega}\lim \limits_{k\rightarrow \infty} h_k(\theta)~d m
(\theta)=\int_{\Omega}\Phi_*(\theta)~d m (\theta),
\end{align*}
where we  use the monotone convergence theorem for  the fourth
equality. Hence we prove (3) in the case that  $\Phi$ is
sub-additive. Now assume that $\Phi$ is asymptotically sub-additive.
For $\epsilon>0$, let $\Psi$ be given as in the first paragraph of
our proof. Then $|\Phi_*(\theta)-\Psi_*(\theta)|\leq \epsilon$ for
any $\theta\in \M(X,T)$. It together with $\Psi_*(\mu)=\int_\Omega
\Psi_*(\theta) ~d m(\theta)$ yields $ \left|\Phi_*(\mu)-\int_\Omega
\Phi_*(\theta) ~d m(\theta)\right|\leq 2\epsilon. $ Letting
$\epsilon\to 0$, we obtain the desired identity for $\Phi$. This
finishes the proof.
\end{proof}

Let $\M(X)$ denote the space of  Borel probability measures on $X$
endowed with the weak-star topology. Then we have
\begin{lem}
\label{lem-3.2} Suppose $\{\nu_n\}_{n=1}^\infty$ is a sequence in
$\mathcal{M}(X)$. We form the new sequence $\{\mu_n\}_{n=1}^\infty$
by $\mu_n=\frac 1n \sum_{i=0}^{n-1}\nu_n\circ T^{-i}$. Assume that
$\mu_{n_i}$ converges to $\mu$ in $\mathcal{M}(X)$ for some
subsequence  $\{n_i\}$ of natural numbers.  Then $\mu\in
\mathcal{M}(X,T)$ and
$$
\limsup_{i\to\infty}\frac{1}{n_i}\int \log
\phi_{n_i}(x)\;d\nu_{n_i}(x)\leq \Phi_*(\mu).
$$
\end{lem}
\begin{proof}
The lemma  was proved in \cite[Lemma 2.3]{CFH08} for the case that
$\Phi$ is sub-additive. Here we shall show that it can be extended
to the case that $\Phi$ is asymptotically sub-additive.

Let $\Phi$ be an asymptotically sub-additive potential on $X$ and
$\epsilon>0$. Then there exist a sub-additive potential $\Psi=\{\log
\psi_n\}_{n=1}^\infty$ on $X$ and $n_0$ such that $|\log
\phi_n(x)-\log \psi_n(x)|\leq n\epsilon$ for any $n\geq n_0$ and
$x\in X$. Hence
$$
\limsup_{i\to\infty}\frac{1}{n_i}\int \log
\phi_{n_i}(x)\;d\nu_{n_i}(x)\leq
\limsup_{i\to\infty}\frac{1}{n_i}\int \log
\psi_{n_i}(x)\;d\nu_{n_i}(x)+\epsilon\leq \Psi_*(\mu)+\epsilon\leq
\Phi_*(\mu)+2\epsilon.
$$
Letting $\epsilon\to 0$, we obtain the desired inequality for
$\Phi$.
\end{proof}

\begin{lem} \label{lem-mf-sa} Define
$\overline{\beta}({\Phi})=\limsup_{n\rightarrow \infty} \sup_{x\in
X}\frac{\log \phi_n(x)}{n}$. Then
\begin{enumerate}
\item $\overline{\beta}(\Phi)\in \mathbb{R}\cup
\{-\infty\}$ and $\overline{\beta}({\Phi})=\liminf_{n\rightarrow
\infty} \sup_{x\in X}\frac{\log \phi_n(x)}{n}$.

\item  $\overline{\beta}(\Phi)=\sup \{ \Phi_*(\mu):\mu\in \mathcal{M}(X,T)\}$ and there
exists an ergodic measure $\nu\in \mathcal{M}(X,T)$ such that
$\overline{\beta}(\Phi)=\Phi_*(\nu)$.

\item The following conditions are
equivalent:

\begin{enumerate}
\item $\overline{\beta}(\Phi)=-\infty$;

\item  $\lambda_{\Phi}(x)=-\infty$ for all $x\in X$;

\item $\Phi_*(\mu)=-\infty$ for all $\mu\in
\mathcal{M}(X,T)$;

\item $P(T,\Phi)=-\infty$.
\end{enumerate}

\item If $\overline{\beta}(\Phi)>-\infty$, then
$\htop(T)+\overline{\beta}(\Phi)\ge P(T,\Phi)\ge
\overline{\beta}(\Phi)>-\infty$. Moreover if we assume in addition
that  $\htop(T)<\infty$, then $P(T,\Phi)\in \mathbb{R}$.
\end{enumerate}
\end{lem}
\begin{proof}  Let $\epsilon>0$. Take a sub-additive potential $\Psi=\{\log
\psi_n\}_{n=1}^\infty$ on $(X,T)$ such that $|\log \phi_n(x)-\log
\psi_n(x)|<n\epsilon$ for all $n\geq n_0$ and $x\in X$. Let
$C=\max_{x\in X} |\psi_1(x)|$. Then $\psi_n(x)\leq C^n$. Thus for
$n\geq n_0$ we have $\log \phi_n(x)\leq n(\log C+\epsilon)$ and
hence $\sup_{x\in X}\frac{\log \phi_n(x)}{n} \le \log C+\epsilon$.
This implies $\overline{\beta}(\Phi)\in \mathbb{R}\cup \{-\infty\}$.
Denote $b_n=\sup_{x\in X}\log \psi_n(x)$. Then by the sub-additivity
of $\Psi$, $b_{n+m}\leq b_n+b_m$. It follows that $\liminf_{n\to
\infty}b_n/n=\limsup_{n\to \infty}b_n/n$ and thus $\liminf_{n\to
\infty}\sup_{x\in X}\log \phi_n(x)/n\geq   \limsup_{n\to
\infty}\sup_{x\in X}\log \phi_n(x)/n-2\epsilon$. Letting
$\epsilon\to 0$, we obtain $$\liminf_{n\to \infty}\sup_{x\in X}\log
\phi_n(x)/n= \limsup_{n\to \infty}\sup_{x\in X}\log \phi_n(x)/n.$$
This proves (1).

 For any $\mu\in \mathcal{M}(X,T)$, by Proposition
\ref{pro-1195}(1),  $$\Phi_*(\mu)=\lim_{n\rightarrow \infty} \int_X
\frac{\log \phi_n(x)}{n} d \mu(x)\le \limsup_{n\rightarrow \infty}
\sup_{x\in X} \frac{\log \phi_n(x)}{n}=\overline{\beta}(\Phi).$$
Hence $\sup\{ \Phi_*(\mu):\mu\in \mathcal{M}(X,T)\}\le
\overline{\beta}(\Phi)$. Conversely, choose $n_i\rightarrow \infty$
and $x_i\in X$  such that $\lim_{i\rightarrow \infty} \frac{\log
\phi_{n_i}(x_i)}{n_i}=\limsup_{n\rightarrow \infty} \sup_{x\in X}
\frac{\log \phi_n(x)}{n}$. Let
$\mu_{n_i}=\frac{1}{n_i}\sum_{j=0}^{n_i-1}\delta_{T^jx_i}$ for $i\in
\mathbb{N}$. Since $\mathcal{M}(X)$ is compact, we may assume that
$\mu_{n_i}\rightarrow \mu$ for some $\mu\in \mathcal{M}(X)$. By
Lemma \ref{lem-3.2}, $\mu\in \mathcal{M}(X,T)$ and
$\lim_{i\to\infty}\frac{1}{n_i}\int \log  \phi_{n_i}(x)\;d
\delta_{x_i}\le \Phi_*(\mu)$, i.e.,
$\overline{\beta}(\Phi)=\lim_{i\to\infty}\frac{\log
\phi_{n_i}(x_i)}{n_i}\le \Phi_*(\mu)$. Moreover, by Proposition
\ref{pro-1195}(3), there exists an ergodic measure $\nu\in
\mathcal{M}(X,T)$ such that $\overline{\beta}(\Phi)\le \Phi_*(\nu)$.
Clearly, $\overline{\beta}(\Phi)=\Phi_*(\nu)$. This proves (2).

To show (3), note that the implications $(a)\Rightarrow (b),(c)$ are
direct.  By (2), there exists an ergodic measure $\nu\in
\mathcal{M}(X,T)$ such that $\overline{\beta}(\Phi)=\Phi_*(\nu)$. By
Proposition \ref{pro-1195}(1),
$\lambda_{\Phi}(x)=\overline{\beta}(\Phi)$ for $\nu$-a.e. $x\in X$.
Hence $\overline{\beta}(\Phi)=-\infty$ when $(b)$ or $(c)$ occurs.
This shows that (b) or (c) implies $(a)$. The equivalence of $(c)$
and $(d)$ comes from Theorem \ref{thm-3.3}. This proves (3). Part
(4) follows directly from (2) and Theorem \ref{thm-3.3}.
\end{proof}

Now we give some properties of asymptotically additive potentials,
which just follow from Proposition \ref{pro-1195}(2) and Lemma
\ref{lem-3.2}.
\begin{lem}
\label{lem-3.4} Assume that $\Phi=\{\log \phi_n \}_{n=1}^\infty$ is
an asymptotically additive potential on $(X,T)$. Then
\begin{itemize}
\item[(i)] The map $\mu\mapsto \Phi_*(\mu)$ is continuous on $\M(X,T)$.
\item[(ii)] Suppose $\{\nu_n\}_{n=1}^\infty$ is a sequence in $\mathcal{M}(X)$. We form the new sequence
$\{\mu_n\}_{n=1}^\infty$ by $\mu_n=\frac 1n
\sum_{i=0}^{n-1}\nu_n\circ T^{-i}$. Assume that $\mu_{n_i}$
converges to $\mu$ in $\mathcal{M}(X)$ for some subsequence
$\{n_i\}$ of natural numbers.  Then $\mu\in \mathcal{M}(X,T)$, and
moreover
$$
\lim_{i\to\infty}\frac{1}{n_i}\int \log
\phi_{n_i}(x)\;d\nu_{n_i}(x)=\Phi_*(\mu).
$$
\item[(iii)] $\Omega:=\{\Phi_*(\mu):\; \mu\in \M(X,T)\}$ is an interval which equals $[\underline{\beta}(\Phi),
\overline{\beta}(\Phi)]$, where
$\underline{\beta}(\Phi):=\lim_{n\to \infty}(1/n)\inf_{x\in X}\log
\phi_n(x)$.
\end{itemize}

\end{lem}

In the end of this section, we give the following proposition.

\begin{pro}
\label{pro-le} Let $\Phi=\{\log \phi_n\}$ be a potential on $X$
(i.e., each $\phi_n$ is a non-negative continuous function on $X$).
We have the following statements.
\begin{itemize}
\item[(i)] If there exists  $C\geq 1$ such that $\phi_{n+m}(x)\leq C\phi_n(x)\phi_m(T^nx)$ for all
$x\in X$ and $n,m\in \N$, then $\Phi$ is asymptotically
sub-additive.
\item[(ii)] If there exists $C\geq 1$  such that
$$0<C^{-1}\phi_n(x)\phi_m(T^nx)\leq \phi_{n+m}(x)\leq C\phi_n(x)\phi_m(T^nx)
$$
for all $x\in X$ and $n,m\in \N$, then $\Phi$ is asymptotically
additive.
\item[(iii)] If $\phi_n(x)>0$ for all $n\in \N$, $x\in X$  and  there exists a continuous function $g$ on $X$
such that $$\log \phi_{n+1}(x)-\log \phi_n(Tx)\to g(x)$$ uniformly
on $X$ as $n\to \infty$, then $\Phi$ is asymptotically additive.
\item[(iv)]$\Phi$ is asymptotically additive if and only if for any $\epsilon >0$, there exists an additive
potential $\Psi=\{\log \psi_n\}_{n=1}^\infty$ on $X$ such that
\begin{equation}
\label{e-1504f} \limsup_{n\to \infty}\frac{1}{n}\sup_{x\in X}|\log
\phi_n(x)-\log \psi_n(x)|\leq \epsilon.
\end{equation}
\end{itemize}

\end{pro}
\begin{proof}
To see (i), define $\Psi=\{\log \psi_n\}_{n=1}^\infty$ by
$\psi_n(x)=C \phi_n(x)$. Then
$$
\psi_{n+m}(x)=C\phi_{n+m}(x)\leq
C^2\phi_{n}(x)\phi_m(x)=\psi_n(X)\psi_m(T^nx),
$$
Hence $\Psi$ is sub-additive. Clearly, $(\log\psi_n(x)-\log
\phi_n(x))/n=(\log C)/n\to 0$ as $n\to \infty$. Hence $\Phi$ is
asymptotically sub-additive. This proves (i).  Part (ii) follows
directly from (i).

To show (iii), define $r_n=\sup_{x\in X} |\log \phi_{n}(x)-\log
\phi_{n-1}(Tx)-g(x)|$, with the convention $\log \phi_0(x) \equiv
0$. It is clear that $\lim_{n\rightarrow \infty} r_n=0$. Let
$g_n=\sum_{i=0}^{n-1}g \circ T^i$. Then $\mathcal{G}=\{
g_n\}_{n=1}^\infty$ is additive. Note that
\begin{align*}
|\log \phi_n(x)-g_n(x)| &=\left|\sum_{i=1}^n
\left(\log \phi_i(T^{n-i}x)-\log \phi_{i-1}(T^{n-i+1}x)-g(T^{n-i}x)\right)\right| \\
&\le \sum_{i=1}^n |\log \phi_i(T^{n-i}x)-\log
\phi_{i-1}(T^{n-i+1}x)-g(T^{n-i}x)|\le \sum_{i=1}^n r_i.
\end{align*}
Hence $\limsup_{n\rightarrow \infty} \sup_{x\in X} |\frac{\log
\phi_n(x)-g_n(x)}{n}|\le \limsup_{n\rightarrow \infty}
\frac{1}{n}\sum_{i=1}^n r_i=0$ as $\lim_{n\rightarrow +\infty}
r_n=0$. Hence  $\Phi$ is asymptotically additive.

The ``if'' part in (iv) is direct, we only need to show the ``only
if'' part. Assume that $\Phi$ is asymptotically additive, that is,
$\phi_n$ is positive continuous on $X$ for each $n$ and both $\{\log
\phi_n\}_{n=1}^\infty$ and $\{\log (\phi_n)^{-1}\}_{n=1}^\infty$ are
asymptotically sub-additive. We claim that for any $\epsilon>0$,
there exists $K>0$ such that for each $k\geq K$, there exists
$C_{\epsilon,k}>0$ so that
\begin{equation*}
\left|\log \phi_n(x)- \frac{1}{k}\sum_{j=0}^{n-1}\log
\phi_k(T^jx)\right|\leq n\epsilon +C_{\epsilon,k}, \qquad \forall\;
n\geq 2k,\; x\in X,
\end{equation*}
 Clearly the above inequality implies the ``only if'' part in (iv).  Without loss of generality, we show that
\begin{equation}
\label{e-1504g}\log \phi_n(x)\leq  \frac{1}{k}\sum_{j=0}^{n-1}\log
\phi_k(T^jx)+ n\epsilon +C_{\epsilon,k}, \qquad \forall\; n\geq
2k,\; x\in X.
\end{equation}
for certain $C_{\epsilon,k}>0$. Fix $\epsilon>0$. Since $\Phi$ is
asymptotically sub-additive,  there exists a
 sub-additive potential $\Psi=\{\log \psi_n\}_{n=1}^\infty$ on $X$ such that there is $K>0$ and
\begin{align}\label{wen-eq-m1}
|\log \phi_n(x)-\log \psi_n(x)|\le \frac{n\epsilon}{2}, \qquad
\forall \;n\ge K,\; x\in X.
\end{align}
Set $C=\max\{1, \sup_{x\in X}\psi_1(x)\}$. By \cite[Lemma 2.2
]{CFH08},
$$\log \psi_n(x)\leq 2k\log C+\frac{1}{k}\sum_{i=0}^{n-k}\log \psi_k(T^ix),\quad \forall\; x\in X,\; n\geq 2k.
$$
Combining the above inequality with (\ref{wen-eq-m1}), we have for
$k\geq K$,
\begin{equation*}
\begin{split}
\log \phi_n(x)&\leq (2n-k)\epsilon/2+ 2k\log C+\frac{1}{k}\sum_{i=0}^{n-k}\log \phi_k(T^ix)\\
&\leq (2n-k)\epsilon/2+ 2k\log C+
M_k+\frac{1}{k}\sum_{i=0}^{n-1}\log \phi_k(T^ix)
\end{split}
\end{equation*}
for all $x\in X$ and $n\geq 2k$, where $M_k:=\max\{1,\sup_{x\in
X}|\log \phi_k(x)|\}$. This proves (\ref{e-1504g}), with
$C_{\epsilon,k}=2k\log C+ M_k$. We finish the proof of the
proposition.
\end{proof}

\begin{re}
\label{re-last}
\begin{itemize}
\item[(i)]The potentials satisfying the assumption in Proposition \ref{pro-le}(iii) was considered by Barreira
\cite{Bar96} in the study of the Hausdorff dimension of planar limit
sets.
\item[(ii)]
Let $\C_{aa}(X,T)$ denote the collection of asymptotically additive
potentials on $X$. Define an equivalence relation $\sim$ on
$\C_{aa}(X,T)$ by $\Phi\sim \Psi$ if $\|\Phi-\Psi\|_{\rm lim}=0$,
where
$$
\|\Phi-\Psi\|_{\rm lim}:=\limsup_{n\to \infty}\frac{1}{n}\sup_{x\in
X}|\log \phi_n(x)-\log \psi_n(x)|
$$
for $\Phi=\{\log \phi_n\}_{n=1}^\infty$, $\Psi=\{\log
\psi_n\}_{n=1}^\infty$. Then it is not hard to see that the quotient
space $\C_{aa}(X,T)/\sim$ endowed with the  norm $\|\cdot\|_{\rm
lim}$ is a separable Banach space.

\end{itemize}
\end{re}

\section{Main notation and conventions}
\label{B}
For the reader's convenience, we summarize in Table \ref{table-1} the main notation and typographical conventions used in this paper.
\begin{table}
\centering
\caption{Main notation and conventions}
\vspace{0.05 in}
\begin{footnotesize}
\begin{raggedright}
\begin{tabular}{p{1.5 in} p{4 in} }
\hline \rule{0pt}{3ex}
$(X,T)$ & A topological dynamical system (Section \ref{S-1})\\
$\Phi=\{\log \phi_n\}_{n=1}^\infty$ & (Asymptotically sub-additive) potential (Section \ref{S-1})\\
$\overline{\beta}(\Phi)$ & $\overline{\beta}(\Phi)=\lim_{n\to \infty}(1/n)\log
\sup_{x\in X}\phi_n(x)$\\
$\lambda_\Phi(x), \Phi_*(\mu)$ & Lyapunov exponent of $\Phi$ at $x$ (resp. with respect to $\mu$) (Section \ref{S-1})\\
$E_\Phi(\alpha)$ & $\alpha$-level set of $\lambda_\Phi$ (Section \ref{S-1})\\
$\M(X)$ & Set of all  Borel probability measures on $X$ \\
$\M(X,T), {\mathcal E}(X,T)$ & Set of $T$-invariant (resp. ergodic)  Borel probability measures on $X$\\
$h_\mu(T)$ & measure-theoretic entropy of $T$ with respect to $\mu$\\
$P(T,\Phi)$ & Topological pressure of $\Phi$ (Section \ref{S-2.2})\\
$P_\Phi(q)$ & $P(T, q\Phi)$\\
$P_\Phi'(\pm \infty)$ & $\lim_{q\to \infty}P_\Phi(q)/q$, $\lim_{q\to -\infty}P_\Phi(q)/q$\\  
$\I(\Phi,q)$ & Set of equilibrium states of $q\Phi$\\
$\htop(T, Z), \htop(T)$ & Topological entropy of $T$ with respect to $Z$ (resp. $Z=X$) (Section \ref{S-2.1})\\
${\rm conv}(M)$ & Convex hull of $M$ (Section \ref{S-2.3})\\
${\rm ri} (A)$ & Relative interior of a convex set $A$ \\
$f^*$ & Conjugate function of $f$ (Section \ref{S-2.4})\\
${\rm ext}(C), {\rm expo} (C) $ & Set of extreme points (resp. exposed points) of $C$ (Section \ref{S-2.3})\\
$\partial f({\bf x})$, &  Subdifferential of $f$ at ${\bf x}$ (Section \ref{S-2.3})\\
$\partial^e f({\bf x})$ & ${\rm ext} (\partial f({\bf x}))$\\
$\partial f(U)$, $\partial^e f(U)$ & $\bigcup_{{\bf x}\in U} \partial f({\bf x})$, $\bigcup_{{\bf x}\in U} \partial^e f({\bf x})$\\
$V(x)$ & Set of limit points of the sequence $\mu_{x,n}=(1/n) \sum_{j=0}^{n-1}\delta_{T^j x}$ in $\M(X)$ \\
$\R_+$ & $(0,\infty)$\\
${\rm cl}_+(A)$ & (cf. \eqref{e-0330a})\\
${\rm cl}_+^\delta(\partial P_{{\bf \Phi}}(\R_+^k))$ & (cf.  \eqref{e-newno})\\
${\bf \Phi}=(\Phi_1,\ldots, \Phi_k)$ & A family of asymptotically sub-additive potentials \\
$\overline\beta({\bf \Phi})$ & $\overline{\beta}(\sum_{i=1}^k\Phi_i)$ \\
${\bf \Phi}_*(\mu)$ & $ ((\Phi_1)_*(\mu),\ldots, (\Phi_k)_*(\mu))$\\
$P_{\bf \Phi}({\bf q})$ &  $P(T, {\bf q}\cdot {\bf \Phi})$\\
$\I({\bf \Phi}, {\bf q})$ & Set of equilibrium states of ${\bf q}\cdot {\bf \Phi}$\\
$E_{{\bf \Phi}}({\bf a})$ &  (cf. \eqref{e-1.4-h})\\
$G_\mu$ & Set of $\mu$-generic points (see Section \ref{S-5})\\

\hline
\end{tabular}
\label{table-1}
\end{raggedright}
\end{footnotesize}
\end{table}

\noindent {\bf Acknowledgements}  The first author was partially
supported by the direct grant and RGC grants in the Hong Kong
Special Administrative Region, China (projects CUHK400706,
CUHK401008).  The second author  was partially supported by NSFC
(Grant 10531010), 973 project and  FANEDD (Grant 200520). The
authors thank Yongluo Cao and Katrin Gelfert for their valuable
comments. They  also thank the anonymous referee for his helpful comments and suggestions that improved the manuscript.

\end{document}